\documentclass[12pt]{article}
\usepackage{amsmath,bm,hyperref,color}
\usepackage{amsfonts}
\usepackage{graphicx}
\usepackage{bm}
\newtheorem{theorem}{Theorem}

\newtheorem{lemma}[theorem]{Lemma}
\newtheorem{proposition}[theorem]{Proposition}

\newenvironment{proof}[1][Proof]{\noindent\textbf{#1.} }{\ \rule{0.5em}{0.5em}}

\begin{document}

\begin{center}
{ \Large
Normal forms for three-parametric families of 
area-preserving maps near an elliptic fixed point\\
}
\end{center}

\begin{center}
Natalia Gelfreikh
\end{center}
\begin{center}
\small Department of Higher Mathematics and Mathematical Physics,
\end{center}

\begin{center}
Saint-Petersburg State University, Russia
\end{center}
\begin{center}
E-mail: st004254@spbu.ru
\end{center}


\begin{abstract}
We study dynamics of area-preserving maps in a neighbourhood
of an elliptic fixed point. We describe simplified normal forms
for a fixed point of co-dimension 3. We also construct normal
forms for a generic three-parameter family 
which contains such degeneracy and use the normal
form theory to describe generic bifurcations
of periodic orbits in these families.

\end{abstract}

\section{Introduction}

In a Hamiltonian system small oscillations around a periodic orbit 
are often described using the normal form theory~\cite{AKN,SM} which provide an important tool for the study of local
dynamics (see
e.g.~\cite{AKN,Bryuno,Kuznetsov}).
 In the case of two degrees of freedom
the Poincar\'e section is used to reduce the problem 
to studying a family of area-preserving maps in a neighbourhood of a fixed point. The Poincar\'e map depends 
on the energy level and possibly on other parameters involved in  the problem. 
A sequence of coordinate changes is used 
to transform the map to a normal form. 
Our approach to the normal form of a map is similar to \cite{Tak74}.

In the absence of resonances 
the normal form is a rotation of the plane, and the angle
of the rotation depends on the amplitude.
In a generic one-parameter family of area-preserving maps,
 the normal form provides a description for a chain of islands which is born from the
origin when the multiplier of the fixed point crosses a resonant 
value~\cite{AKN,SM,Meyer1970,SV2009,GG2009}.

In \cite{GG2014}
unique normal forms for two-parametric families were constructed and used to analyse bifurcations of $n$-periodic orbits.  
 
In the paper we study three-parametric families of APM 
with fixed point of elliptic type.
In such families two possibilities
of degenerations are possible:

(1) degenerations in the twist terms;

(2) degeneration in the leading resonance term.

Normal form in the case of (1) was constructed in \cite{GG2014}
(for arbitrary number of parameters). Normal forms for 
(2) are constructed in the paper first for an individual map (Section~\ref{Se:hn0is0}) and then for families
(Section~\ref{Se:Familieshn0}). In Sections~\ref{Se:biflead}
and  \ref{Se:Familiesh33} the normal forms are used 
to inverstigate biffurcations.

\subsection{Individual maps}

Let $F_0:\mathbb{R}^{2}\to \mathbb{R}^{2}$ be an area-preserving map (APM) which also
preserves orientation. Let the origin be a fixed point:
$
F_0(0)=0.
$
Since $F_0$ is area-preserving $\det DF_0(0)=1$. Therefore
the two eigenvalues of the Jacobian matrix  
$DF_0(0)$ are $\lambda_0 $ and $\lambda_0 ^{-1}$.
We will consider an elliptic fixed point, i.e. the case 
  of non real $\lambda_0$.
As the map is real $\lambda^{-1}_0=\lambda_0^*$.
Consequently the multipliers of an elliptic fixed point
belong to the unit circle: $|\lambda_0|=1$, i.e. $\lambda_0 =
e^{i\alpha_0}$. 

There is a linear area-preserving change of variables
such that the Jacobian of $F_0$ takes the form of
a rotation:
\begin{equation}\label{Eq:rotation}
DF_0(0)=R_{\alpha_0 },\qquad\mbox{where }R_{\alpha_0}=\left(
\begin{array}{cc}
\cos \alpha_0 & -\sin \alpha_0 \\
\sin \alpha_0 & \cos \alpha_0%
\end{array}%
\right).
\end{equation}

It is well known that APM with elliptic fixed point can be
represented in Birkhoff normal form \cite{AKN}, i.e.
there 
is an area-preserving change of coordinates which transforms $F_0$ into the resonant normal form
$N_0$ that commutes with the rotation: $N_0\circ R_{\alpha_0}=R_{\alpha_0}\circ N_0$. The change 
of coordinates and the map $N_0$ are formal series. 
The linear part of a normal form $N_0$ is $R_{\alpha_0}$. Following  the
method suggested in \cite{Tak74} by Takens (see e.g. \cite{BGSTT1993,GG2009}), we 
consider a formal series $H_0$ such that
\begin{equation}
N_0=R_{\alpha_0}\circ \Phi^1_{H_0},
\end{equation}
where $\Phi_{H_0}^t$ is a flow generated by the Hamiltonian $H_0$. The Hamiltonian 
has the Takens normal form, i.e. it is invariant
with respect to the rotation: $H_0\circ R_{\alpha_0}=H_0$.

Our goal is to transform formal series
$H_0$ to the most simple form.
We use  changes of variables
which commutate with 
the rotation $ R_{\alpha_0}$. Then in new variables
the map is still in Birkhoff normal form and
the corresponding Hamiltonian remains in 
the Takens normal form.

It is convinient to use complex variables defined by
\begin{equation}\label{Eq:z}
z=\frac{x+iy}{\sqrt2} \qquad\mbox{and}\qquad\bar z=\frac{x-i y}{\sqrt2}.
\end{equation}

A fixed point is called {\em resonant} if there exists $n\in \mathbb{N}$ such that
$\lambda_0^{n}=1$. The least positive $n$ is called the {\em order of the resonance} \cite{AKN}.
The rotation $R_{\alpha_0}$ takes the form 
$(z, \bar z) \mapsto (e^{i\alpha_0}z,e^{-i\alpha_0}
\bar z )$.
As the map $N_0$ commutes with $R_{\alpha_0}$
it contains only resonant terms:
$$
N_0(z,\bar z)=\lambda_0 z
+\sum_{\substack {k+l \ge 2 , \\k-l=1 \pmod n}} f_{kl} z^k \bar z^l .$$
 Corresponding Hamiltonian has the Takens normal form
 \cite{GG2009}:  
\begin{equation}\label{Eq:hinitial} 
H_0(z,\bar z)=\sum_{\substack {k+l \ge 3 , \\k=l \pmod n}}
h_{kl} z^k \bar z^l , \qquad h_{kl}=h_{lk}^*.\end{equation}

It was established in \cite{GG2009}, \cite{GG2014} that if
 $h_{n0} \ne 0$ 
then Hamiltonian $H_0$ can be transformed to a normal form
 
\[ \widetilde H_0(z,\bar z)=\sum_{k \ge 2} a_k z^k \bar z^k
+(z^n+\bar z^n)\sum_{k \ge 0} b_k z^k \bar z^k .\]

In the paper we consider the case of $h_{n0} = 0$ but 
$h_{22}\ne 0$ and $h_{n+1,1}^2-4h_{22}h_{2n,0} \ne 0$.

In order to investigate Hamiltonian it is convenient to use the symplectic polar coordinates
$(I,\varphi)$ given by
\begin{equation}
\label{Eq:zpolar}
\left\{
\begin{array}{rcl}
x&=&\sqrt{2I}\cos\varphi\,,\\
y&=&\sqrt{2I}\sin\varphi\,
\end{array}
\right. \quad {\rm or} \quad
\left\{
\begin{array}{rcl}
z&=&\sqrt{I}e^{i\varphi}\,,\\
\bar z&=&\sqrt{I}e^{-i\varphi}\,.
\end{array} \right.
\end{equation}
The Hamiltonian in Takens normal form (\ref{Eq:hinitial}) takes the form :
\begin{equation}\label{Eq:inter}
H_0(I,\varphi)=I^2\sum_{k\ge 0} a_kI^k+\sum_{j\ge 1,k\ge 0} a_{jk}I^{jn/2+k}\cos (j n\varphi +\beta_{jk}),
\end{equation}
where $a_k$,  $a_{jk}$  and $\beta_{jk}$ are real coefficients. 

Proposition~\ref{Pro:hn0is0} of Section~\ref{Se:hn0is0}
implies the following theorem.
\begin{theorem}\label{Te:N0}
Let $N_0$ be BNF for some APM $F_0$ with resonance of order $n$ at the origin. Let for corresponding 
Hamiltonian in TNF (\ref{Eq:inter}) $a_{10}=0$, $a_0\ne 0$ and
$a_{11}^2e^{i2\beta_{11}}-8a_0a_{20}e^{i\beta_{20}}\ne 0$.
Then there is a formal canonical change of coordinates such that
in new variables
$\widetilde N_0 =R_{\alpha_0}\circ \Phi^1_{\widetilde H_0}$ 
and
\begin{equation}\label{Eq:Hn0is0}
\widetilde H_0(I,\varphi )=I^{2}\sum_{k\ge 0} a_k {I}^k    
+I^{n/2} \sum_{k\ge 1}  b_k  \cos (n\varphi + \psi_k )  {I}^{2k} +I^{n} \cos 2n\varphi \sum_{k\ge 0}c_k I^{2k}.
\end{equation}
\end{theorem}

The coefficients in the form (\ref{Eq:Hn0is0})
are not unique: the coefficients $b_k$ are replaced  by
$-b_k$ after rotation by $\pi /n$.

There is also an alternative normal form 
(only $a_0\ne 0$ is required),
which contains fewer terms of low orders:
\begin{theorem}\label{Te:N01}
Let $N_0$ be BNF for some APM $F_0$ with resonance of order $n$ at the origin. Let for corresponding 
Hamiltonian in TNF (\ref{Eq:inter})  $a_0\ne 0$.
Then there is a formal canonical change of coordinates such that
in new variables
$\widetilde N_0 =R_{\alpha_0}\circ \Phi^1_{\widetilde H_0}$ 
and
\[
 {\widetilde H_0}(I,\varphi )=I^{2}\sum_{k\ge 0} a_k {I}^k    
+\sum_{k\ge 1} c_k I^{nk/2} \cos (kn\varphi +\psi_k). \]
\end{theorem}

Corresponding Proposition~\ref{Pro:hn0is0case1} 
is proved in Section~\ref{Se:hn0is0}.


\subsection{Families}

Let us consider a three-parametric family $F_\mu$ of APM with fixed point
of an elliptic type at the origin with 
$\lambda_\mu =e^{i\alpha_\mu}$, 
$ \mu =(\mu_1,\mu_2,\mu_3) \in \mathbb{R}^3$.
We assume that for $\mu = (0,0,0)$ 
the function $F_0$ has resonance of order $n$ (i.e. 
$\lambda_0=e^{i\alpha_0}$, $\lambda_0^n=1 $).

After a linear change of coordinates the map $F_\mu$
takes the form
$ F_\mu= R_{\alpha_\mu} \circ \Phi $,
where $\Phi$ is a tangent-to-identity APM.
In complex variables (\ref{Eq:z}) $F_\mu =(f_\mu ,\bar f_\mu)$,
\[ f_\mu (z,\bar z)=\lambda_\mu z +\sum_{k+l\ge 2}
f_{kl}(\mu )z^k\bar z^l .\]
It is natural to use the value
\begin{equation}\label{Eq:epsilon}
\varepsilon = \alpha_0 -\alpha_\mu
\end{equation}  
as one of parameters.
Let $\mu_3$  be expressed in terms of 
$(\varepsilon , \mu_1, \mu_2 )$. Then
$f_\mu(z,\bar z)$ can be presented as a series in
three variables $(z,\bar z, \varepsilon )$
with coefficients depending on $(\mu_1 ,\mu_2 )$:
\[ f_\mu (z,\bar z)=\lambda_0 z+\sum_{k+l+m \ge 2}
f_{klm}(\mu_1,\mu_2)z^k\bar z^l \varepsilon^m. \]
After an appropriate change of coordinates the map
$f_\mu$ can be written in
Birkhoff  normal form which contains only resonant terms:
\begin{equation}\label{Eq:BNFfam}
f_\mu (z,\bar z)=\lambda_0 z+\sum
_{\substack{ k+l+m \ge 2\cr k-l=1\pmod n}}
f_{klm}(\mu_1,\mu_2)z^k\bar z^l \varepsilon^m.
\end{equation}

Interpolation theorem for families (see for example \cite{GG2009})
gives $F_\mu=R_{\alpha_{0}} \circ \Phi^1_{H_\mu} $, 
where  Hamiltonian $H_\mu$ in complex variables 
has the form
\begin{equation}\label{Eq:hpinitial} 
H_\mu (z, \bar z)=\sum_{\substack{ k+l+m \ge 3\cr k=l\pmod n}}h_{klm}(\mu_1, \mu_2)z^k\bar z^l\varepsilon^m ,
\end{equation}
where $h_{111}(\mu_1, \mu_2) =1$
as $\varepsilon $ is determined by (\ref{Eq:epsilon}).
Or, in the symplectic coordinates (\ref{Eq:zpolar})

\begin{equation}\label{Eq:hpinitialIPhi} 
\begin{array}{rcl}
H_\mu (I,\varphi )&=&\varepsilon I +
I^2\sum_{\substack{ k \ge 0\cr m \ge 0}}a_{km}(\mu_1, \mu_2)I^k\varepsilon^m  \\
&+&\sum_{\substack{ j \ge 1\cr k,m \ge 0}}
a_{jkm}(\mu_1, \mu_2)I^{k+j\frac{n}{2}}\varepsilon^m
\cos (jn\varphi +\psi_{jkm}(\mu_1, \mu_2)).
\end{array}
\end{equation}

The form 
(\ref{Eq:hpinitialIPhi})
is not unique. Rotations and time-one shifts 
$\Phi^1_\chi$ with resonant Hamiltonian $\chi $
preserve its structure.
Our goal is to derive the most simple form of $H_{\mu}$,
{\it i.e.} to eliminate as many terms of $H_{\mu}$
as possible. If $a_{100}(\mu_1 ,\mu_2 )=
2|h_{n00}(\mu_1 ,\mu_2 )|\ne 0$ for all values of
$(\mu_1 ,\mu_2 )$ then
 Hamiltonian can be transformed 
to the form \cite{Gelfarxiv}:
\[\widetilde H_\mu(z, \bar z)=\varepsilon z \bar z +
\sum_{ k \ge 2, m \ge 0}a_{km}(\mu_1,\mu_2)z^k\bar z^k\varepsilon^m 
+ (z^n+\bar z^n)\sum_{k,m \ge 0}b_{km}(\mu_1,\mu_2)
z^k\bar z^k\varepsilon^m\]
or, in coordinates $(I,\varphi )$:
\begin{equation}\label{Eq:Hinpolar}
\widetilde H_\mu (I,\varphi )
=\varepsilon I +
\sum_{k \ge 2, m\ge 0}a_{km}(\mu_1, \mu_2)I^k\varepsilon^m 
+I^{n/2}\cos n\varphi \sum_{k,m \ge 0}b_{km}(\mu_1, \mu_2)I^k
\varepsilon^m.
\end{equation}

In this paper the main attention is paid to the case
of Hamiltonian (\ref{Eq:hpinitial}) 
when the following conditions are valid:
\begin{equation}\label{Eq:condit}
h_{n00}(0,0)=0, \quad
h_{220}(0,0)\ne 0,\quad
h_{n+1,1,0}^2(0,0)-4h_{220}(0,0)h_{2n,0,0}(0,0)\ne 0.
\end{equation}
Note that these conditions correspond to a family which is a
three-parametric unfolding of $F_0$ from Theorem~\ref{Te:N0}. We show that there is such canonical change of
variables that Hamiltonian has the form (\ref{Eq:hpinitialIPhi}) with
 the last sum 
containing only two harmonics, {\it i.e.} 
terms with $j=1$ and $j=2$.
Moreover $a_{1,2l+1,m}=0$, $a_{2,2l+1,m}=0$
and $\psi_{2,2l,m}=0$ for all $l$ and $m$.

For the sake of receiving 
$\psi_{200}(\mu_1,\mu_2)=0$, 
as it will be shown
in Section~\ref{Se:Familieshn0}, it is necessary 
to make a rotation by an angle $\varphi =-\frac{1}{2n}
\arg \left( h_{2n,0,0}-\frac{h_{n+1,1,0}^2}{4h_{220}}   \right)  $.

After the rotation the small coefficient of $z^n$
in (\ref{Eq:hpinitial}) is
\begin{equation}\label{Eq:nu}
h_{n00} (\mu_1, \mu_2)\exp \left( -\frac{i}{2}
\arg \left( h_{2n,0,0}(\mu_1, \mu_2)-\frac{h_{n+1,1,0}^2(\mu_1, \mu_2)}{4h_{220}(\mu_1, \mu_2)}   \right) \right)=\nu(\mu_1, \mu_2) ,
\end{equation}
\[ \nu (\mu_1, \mu_2)= \nu_1(\mu_1, \mu_2)+i\nu_2(\mu_1, \mu_2)
=\gamma (\mu_1, \mu_2)e^{i\beta(\mu_1, \mu_2)}.\]
Let $\mu_1$
and $\mu_2$ be expressed by $\nu_1$ and $\nu_2$:
$\mu_1=\mu_1 (\nu_1,\nu_2)$, $\mu_2=\mu_2 (\nu_1,\nu_2)$.
Then we can consider $(\nu_1,\nu_2)$ as new parameters
instead of $(\mu_1,\mu_2 )$.

Below we use the following notations: 
$\mu=(\mu_1, \mu_2, \mu_3)$,  $\Upsilon =(\varepsilon, \nu_1,\nu_2   ) $, $\mathbf{m}=(m_1,m_2,m_3)$, $|\mathbf{m}|=m_1+m_2+m_3$, $\Upsilon^{\mathbf{m}}= \varepsilon^{m_1} \nu_1^{m_2}\nu_2^{m_3}$.

The following theorem gives a simplification of the Hamiltonian to a normal form.
\begin{theorem}\label{Thr:Famhn0is0}
Let $F_{\mu}$ be a smooth ($C^\infty$ or analytic) family of area preserving maps with fixed point
of elliptic type at the origin such that

(1) $F_{\mathbf{0}}$ has a resonant of order $n$
at the origin: $\lambda_0=e^{i\alpha_0}$, 
$\lambda^n=1$;

(2) coefficients of Hamiltonian $H_\mu$
in the TNF (\ref{Eq:hpinitial}) satisfy conditions 
(\ref{Eq:condit});

(3) parameters $\mu =(\mu_1,\mu_2,\mu_3)$
can be expressed by $(\varepsilon , \nu_1, \nu_2)$
defining by
(\ref{Eq:epsilon}) and (\ref{Eq:nu}).

Then there is a formal Hamiltonian 
$\widetilde H_{\Upsilon}$ and
formal canonical change of variables which conjugates $F_{\mu}$ with
$R_{\alpha_0}\circ\Phi^1_{\widetilde H_{\Upsilon}}$.
Moreover, $\widetilde H_{\Upsilon}$
in coordinates $(I,\varphi )$ has the 
following form:
\begin{eqnarray*}
\widetilde H_{\Upsilon}(I, \varphi  )=\varepsilon I+\gamma I^{n/2} \cos (n\varphi +\beta )
+I^{2}\sum_{k,|\mathbf{m}|\ge0} a_{k\mathbf{m}} I^k \Upsilon^{\mathbf{m}}    
\\ +I^{n/2}\sum_{\substack {k,|\mathbf{m}|\ge0, \\k+m_1\ge 1}} b_{k\mathbf{m}}{I}^{2k}  \Upsilon^{\mathbf{m}}  \cos (n\varphi +\psi_{k\mathbf{m}} )   
\\ + I^n \sum_{k,|\mathbf{m}|\ge0} c_{k\mathbf{m}}I^{2k}\Upsilon^{\mathbf{m}}\cos (2n\varphi)   .
\end{eqnarray*}
\end{theorem}
The theorem follows from Proposition~\ref{Pro:Famhn0is0} of Section~\ref{Se:Familieshn0}.

From Proposition~\ref{Pro:Famhnois0} follows an alternative variant, useful for study of bifurcations:

\begin{theorem}\label{Thr:Famhnois0}
If $F_{\mu}$ is a smooth ($C^\infty$ or analytic) family of area preserving maps
such that $F_{\mathbf{0}}$ has a  resonant
 elliptic fixed point at the origin and in
  (\ref{Eq:hpinitialIPhi}) $a_{00}(0,0)\ne 0$,
then there is a formal Hamiltonian 
$\widetilde H_{\Upsilon}$ and
formal canonical change of variables which conjugates $F_\mu$ with
$R_{\alpha_0}\circ\Phi^1_{\widetilde H_{\Upsilon}}$.
Moreover, $\widetilde H_{\Upsilon}$ 
has the following form:
\begin{equation}\label{Eq:Hfamalter}
\begin{array}{r}
\widetilde H_{\Upsilon}(I, \varphi  )=\varepsilon I+\gamma I^{n/2} \cos (n\varphi +\beta )
+I^{2}\sum_{k,|\mathbf{m}|\ge0} a_{k\mathbf{m}} I^k \Upsilon^{\mathbf{m}} \\
\qquad \qquad + I^{n/2} 
\sum_{\substack {k+m_1\ge 1, \\m_2,m_3\ge 0}}
c_{k\mathbf{m}}I^{kn/2}\Upsilon^{\mathbf{m}}\cos ((k+2)n\varphi+\psi_{k\mathbf{m}})  .
\end{array}
 \end{equation}

\end{theorem}

Then we use the normal forms (\ref{Eq:Hinpolar}) and 
(\ref{Eq:Hfamalter}) to study biffurcations.
We keep lower order terms 
and skip those, which do not change the picture qualitatively.
The detailed discussion of typical level sets and their bifurcations for the case of degeneracy in the main
resonant term are presented in Section~\ref{Se:biflead}.

In Section~\ref{Se:Familiesh33} there is a brief 
description of some possible bifurcations for
three-parametric families
when in (\ref{Eq:Hinpolar}) $a_{20}(0,0)=a_{30}(0,0)=0$, 
 $b_{00}(\mu_1 ,\mu_2 )\ne 0$.
This case is different from two-parametric families considered in \cite{GG2014} only in tiny domain in the space of parameters. 

\section{Simplification of a formal Hamiltonian with $h_{n,0}=0$ for individual map\label{Se:hn0is0}} 

In this section we construct two degenerate resonant normal
forms (\ref{Eq:NFhn0is0}) and  (\ref{Eq:hnois0}).
The first one provides
particularly simple form of Hamiltonian in the symplectic polar coordinates: it contains only two lowest harmonics of
the angle variable, namely $n\varphi$ and $2n\varphi$.
An alternative normal form (\ref{Eq:hnois0}) 
contains fewer terms of low orders.

\begin{proposition}\label{Pro:hn0is0}
If  \begin{equation}\label{Eq:hindiv}
H(z,\bar z)=\sum_{\substack {k+l \ge 4 ,\ k,l\ge 0, \\k=l \pmod n}}  h_{kl}z^k \bar z^l 
\end{equation}
is a formal series such that 
$h_{kl}=h_{lk}^*$, $h_{n0}=0$, $h_{22} \ne 0$ and
$h_{n+1,1}^2-4h_{22}h_{2n,0}\ne  0$ 
then there exists a formal canonical change of variables which
transforms the Hamiltonian $H$
into
\begin{equation}\label{Eq:NFhn0is0}
\widetilde H(z,\bar z)=\sum_{k\ge 0} a_k {(z\bar{z})}^{k+2}
+\sum_{k\ge 1}(b_k z^{n}+b_k^*\bar z^n) {(z\bar{z})}^{2k}
 +\sum_{k\ge 0} c_k {(z\bar{z})}^{2k} (z^{2n}+\bar z^{2n}),
\end{equation}%
where $a_k, c_k \in \mathbb{R}$, $b_k \in \mathbb{C}$.
Moreover, $a_0=h_{22}$ and 
$c_0=\left|h_{2n,0}-\frac{h_{n+1,1}^2}{4h_{22}}\right|$.
\end{proposition}

\begin{proof} The key idea behind the proof is based on studying terms
of formal power series in the order of their $\delta$-degree. For a resonant monomial ($z^kz^l$, $k=l \mod n $) we define its $\delta$-degree
by
\begin{equation}\label{Eq:deltahn0is0}
\delta(z^k\bar z^l)=
\left|\frac{k-l}n\right|+\min\{\,k,l\,\}
=
\frac{1}{2}(k+l)-\frac{n-2}{2n}|k-l|\,.
\end{equation}

Grouping in (\ref{Eq:hindiv}) terms of the same $\delta$-degree we get
\[H(z,\bar z)=\sum_{m\ge 2} h_m(z,\bar z), \]
where $h_m(z,\bar z)$ is a homogeneous 
resonant polynomial of 
$\delta$-degree $m$. 

The term of the lowest $\delta$-degree is
\[ h_2(z,\bar z)=h_{22}z^2\bar z^2 +h_{n+1,1}z^{n+1}\bar z+h_{1,n+1}z\bar z^{n+1} +h_{2n,0}z^{2n}+h_{0,2n}\bar z^{2n}.\]

Let $\chi$ be a resonant polynomial. After the  substitution
$(z, \bar z ) \to \Phi^1_{\chi}(z, \bar z )$ the Hamiltonian takes the form
\begin{equation}\label{Eq:htransL}
\widetilde H=H+L_{\chi}H+\sum_{k\ge 2}\frac{1}{k!}L^k_{\chi}H\,,
\end{equation}
where
\[
L_{\chi} H=-i\{ H, \chi \} .
\]

The $\delta$-degrees of monomials in Poisson brackets
\[
\{z^{k_1}\bar z^{l_1},z^{k_2}\bar z^{l_2}\}
=(k_1l_2-k_2l_1)z^{k_1+k_2-1}\bar z^{l_1+l_2-1}
\]
satisfy the following relation:
\begin{equation}\label{Eq:Pbdelta}
\delta(z^{k_1+k_2-1}\bar z^{l_1+l_2-1})
\ge \delta (z^{k_1}\bar z^{l_1})
+\delta(z^{k_2}\bar z^{l_2})-1.
\end{equation}

Indeed,
\[ \begin{array}{l}
\delta(z^{k_1+k_2-1}\bar z^{l_1+l_2-1})=
\frac{1}{2}(k_1+k_2+l_1+l_2-2)-
\frac{n-2}{2n}|k_1-l_1+k_2-l_2|\\
\ge \frac{1}{2}(k_1+l_1)-\frac{n-2}{2n}|k_1-l_1|
+\frac{1}{2}(k_2+l_2)-\frac{n-2}{2n}|k_2-l_2|-1.
\end{array}
\]

Let $\chi=\alpha z^n+\alpha^*\bar z^n$
whith $\alpha=
- \dfrac{h_{n+1,1}}{2inh_{22}}$.
It is not difficult to see that
after the  substitution
$(z, \bar z ) \mapsto \Phi^1_{\chi}(z, \bar z )$
 the coefficients
of Hamiltonian
$\widetilde H=\sum_{k+l\ge 4}\tilde h_{kl}z^k\bar z^l$ of
$\delta$-degree 2 are
\[\left\{ \begin{array}{l} 
\tilde h_{22}=h_{22}, \\
\tilde h_{n+1,1}=h_{n+1,1}+2in\alpha h_{22}=0,\\
\tilde h_{2n,0}=
h_{2n,0}+in\alpha h_{n+1,1}-n^2\alpha^2h_{22}
=h_{2n,0}-\frac{h_{n+1,1}^2}{4h_{22}}. 
\end{array} \right. 
\]

After rotation $z \mapsto e^{-i\frac{\arg \tilde h_{2n,0}}{2n}}z$
the coefficient $\tilde h_{2n,0}\mapsto c_0= |\tilde h_{2n,0}|$.
Then in new variables
\begin{equation}\label{Eq:h2nf}
h_2(z,\bar z)=a_0z^2\bar z^2 +c_0(z^{2n}+\bar z^{2n}),
\end{equation}
where $a_0=h_{22}$, $c_0=\left|h_{2n,0}-\frac{h_{n+1,1}^2}{4h_{22}}\right|$. So the terms of $\delta$-degree 2
has the declared form.

We proceed by induction.
Let $h_m(z,\bar z)$ has the declared form for $m<p$, namely
\begin{eqnarray}\label{Eq:h_m}
h_{2j+2}(z,\bar z)=a_{2j}{(z\bar z)}^{2j+2}+c_j{(z\bar z)}^{2j}(z^{2n}+\bar z^{2n}),\quad {\rm for\ }0\le j \le 
\left\lfloor \frac{p-3}{2}\right\rfloor ,\\
h_{2j+1}(z,\bar z)=a_{2j-1}{(z\bar z)}^{2j+1}+{(z\bar z)}^{2j}(b_jz^{n}+b_j^*\bar z^{n}),\quad {\rm for\  }1\le j \le 
\left\lfloor \frac{p-2}{2}\right\rfloor .
\end{eqnarray}

For the sake of receiving such form for the term of 
$\delta$-degree $p$ we use several changes
of variables consequently.
Any homogeneous polynomial $\chi $ of $\delta$-degree $p-1\ge 2$ generates 
the change of variables $(z, \bar z ) \mapsto \Phi^1_{\chi}(z, \bar z )$. Formulae (\ref{Eq:htransL}) and (\ref{Eq:Pbdelta}) imply
\[\tilde h_m=h_m\quad {\rm for\  } m\le p-1 \]
and for $m=p$
\[ \tilde h_p=h_p+ L(\chi ).\] 
Here $L(\chi )$ is the homological operator: 
\begin{equation}\label{Eq:homop}
L(\chi )=\left[  L_\chi h_2 \right]_p ,
\end{equation} 
where $h_2$ is determined by (\ref{Eq:h2nf})
and $[\cdot ]_p$ denotes terms of 
$\delta$-degree $p$.
 
Let $p$ be $\delta$-degree of resonant monomial $z^k\bar z^l$ and $j=\frac{k-l}{n}$. Let
\[ Q_{p0}(z,\bar z)=z^p\bar z^p\]
and for $1\le j\le p$
\[ Q_{pj}(z,\bar z)=z^{p+j(n-1)}\bar z^{p-j}, \quad
Q_{p,-j}(z,\bar z)=z^{p-j}\bar z^{p+j(n-1)}.\]
The homological operator (\ref{Eq:homop}) acts on monomial 
$Q_{p-1,j}(z,\bar z)$
of $\delta$-degree $p-1$ with $j>0$ by
\[L(Q_{p-1,j})=2ia_0njQ_{pj}
-2inc_0(p-1-j)Q_{p,j+2}  \]
and for $j=0$:
\[ L(Q_{p-1,0})=-2inc_0(p-1)
(Q_{p,2}-Q_{p,-2}).
\]

We denote coefficients of resonant polynomials
$h_p(z,\bar z)$ and $\tilde h_p (z,\bar z)$
by $g_j$ and $\tilde g_j$ respectively:
\[ h_p=g_0Q_{p0}+\sum_{1\le j\le p}
(g_jQ_{pj}+g_j^*Q_{p,-j}),
\quad  \tilde h_p=\tilde g_0Q_{p0}+\sum_{1\le j\le p}
(\tilde g_jQ_{pj}+\tilde g_j^*Q_{p,-j}).\]

The change of variables generated by resonant polynomial
$\chi_k$ of $\delta$-degree $p-1$
\begin{equation}\label{Eq:chij}
\begin{array}{lr}
\chi_k=\alpha_kQ_{p-1,k}+\alpha_k^*Q_{p-1,-k} ,\,& 1 \le k \le p-1, \\
\chi_0=\alpha_0Q_{p-1,0}\,  \end{array} 
\end{equation}
transforms coefficients $g_j$ by the following way:
\[ \left\{
\begin{array}{l}
\tilde g_j = g_j, \qquad j\notin \{ k,k+2 \},
\\ \tilde g_k =g_k +2ia_0nk\alpha_k ,\\
\tilde g_{k+2}=g_{k+2}-2inc_0(p-1-k)\alpha_k .
\end{array} \right. \]
Below we describe the order of this transformations
wich provides the declared form of the terms of $\delta$-degree $p$.

Let $k_0=1$ for even $p$ and $k_0=2$ for odd $p$.
Change of variables generated by $\chi_k$  (\ref{Eq:chij}) 
with $\alpha_k=-\frac{g_k}{2ia_0kn}$
eliminates $g_k$, changes $g_{k+2}$ and does not affect 
any other terms of $\delta$-degree $p$:
\begin{equation*}
\left\{
 \begin{array}{l}
 \tilde g_j=g_j \, ,\quad j\not\in \{ k,k+2 \} ,\\
\tilde g_k=0 , \\
\tilde g_{k+2}=g_{k+2} + \frac{c_0(p-1-k)}{ka_0}g_k\, 
. \end{array} 
\right.
\end{equation*}
Starting with $k=k_0$, then $k=k_0+2$ {\it etc.} 
proceeding  up to $k=p-1$ one get corresponding
terms $\tilde g_k=0$ for all $k\ge k_0$ for which
$k-k_0$ is even.

To eliminate terms with odd $k-k_0$ we use 
changes of variables generated by $\chi_k$ (\ref{Eq:chij}) 
with
\begin{equation*}\begin{array}{lr}
\alpha_k= \frac{g_{k+2}}{2inc_0(p-k+1)} ,\,& 1\le k \le p-2, \\
\alpha_0=\frac{\Im g_2}{2nc_0(p-1)}\, . \end{array} 
\end{equation*}
Each of these substitutions eliminates $g_{k+2}$
(for $k=0$ only image part of $g_2$ eliminates),
changes $g_k$ and does not affect 
any other terms of $\delta$-degree $p$:
\begin{equation*}
 \left\{\begin{array}{l}
\tilde g_j=g_j \, ,\quad j\not\in \{ k,k+2 \},\\ \tilde g_{k+2}=0 ,\, \quad ({\mbox{for} \ } k=0 \quad 
\Im \tilde g_2=0),\\
\tilde g_k=g_{k} + \frac{ka_0g_{k+2}}{c_0(p-k+1)}\, 
 .\end{array} 
\right.
\end{equation*}
Starting from $k=p-2$, then $k=p-4$ {\it etc.} down to
$k=k_0-1$ we get corresponding $\tilde g_{k+2}=0$ and 
$\tilde g_2 \in\mathbb{R}$.
\end{proof}

{\bf Remark about uniqueness.}
The kernel of the homological operator is empty for odd $p$ and one-dimensional for even $p$.
So only $[h_2^k]_{2k} \in \ker L $ which implies uniqueness of the coefficients $a_k$, $c_k$, $|b_k|$ and 
$\arg b_k \pmod \pi $ in formula (\ref{Eq:NFhn0is0}). 
But rotation by $\pi /n$  
change $\arg b_k \mapsto \arg b_k +\pi$.


An alternative normal form is derived in the next 
proposition.
\begin{proposition} \label{Pro:hn0is0case1}
If  
\begin{equation}\label{Eq:h_original1}
H(z,\bar z)=a_0z^2\bar z^2+
\sum_{\substack {k+l \ge 5 ,\ k,l\ge 0, \\k=l \pmod n}}  h_{kl}z^k \bar z^l 
\end{equation}
is a formal series such that $h_{kl}=h_{lk}^*$
and $a_0\ne 0$,
then there exists a formal tangent-to-identity canonical change of variables which
transforms the Hamiltonian $H$
into
\begin{equation}\label{Eq:hnois0}
\widetilde H(z,\bar z)=(z\bar{z})^{2}\sum_{k\ge 0} a_k {(z\bar{z})}^k
+\sum_{k\ge 1} (c_k  z^{nk}+c_k^*\bar z^{nk}),
\end{equation}%
where $a_k \in \mathbb{R}$, 
$c_k\in\mathbb C$.
\end{proposition}

\begin{proof} 
Let $\chi (z, \bar z )=\alpha z^k \bar z^l
+\alpha^*z^l \bar z^k$, $k > l$.  After the  substitution
$(z, \bar z ) \to \Phi^1_{\chi}(z, \bar z )$ the Hamiltonian takes the form (\ref{Eq:htransL}),
where
\[
L_{\chi} H=-i\{ H, \chi \} 
=2ia_0\alpha (k-l)z^{k+1}\bar z^{l+1}
-2ia_0\alpha^* (k-l)z^{l+1}\bar z^{k+1} +O_{k+l+3},
\]
where $O_{k+l+3}$ denotes terms of degree ${k+l+3}$
and higher.
Then
$\tilde h_{st}=h_{st}$  for $ 4\le s+t\le k+l+1$ and 
for $s+t=k+l+2$ if $(s,t)\ne (k+1,l+1)$.
And
\[
\tilde h_{k+1,l+1}=h_{k+1,l+1}+2ia_0\alpha (k-l) \,.
\]

Let $\alpha = - \dfrac{h_{k+1,l+1}}{2ia_0(k-l)}$. 
Then $\tilde h_{k+1,l+1}=0$.

Repeating this substitutions one gets by induction $\tilde h_{kl}=0$ with the exception of $\tilde h_{kk}$ and $\tilde h_{nk,0}=\tilde h^*_{0,nk}$.

\end{proof}


\section{Simplification of a formal Hamiltonian for families with small $h_{11}$ and $h_{n0}$\label{Se:Familieshn0}} 

Now we consider the three-parametric unfolding of Hamiltonian (\ref{Eq:hindiv}):
\begin{equation}\label{Eq:hfam}
H(z,\bar z; \varepsilon ,\mu_1,\mu_2 )=
\varepsilon z\bar z+
\sum_{\substack{k+l\ge 3, m \ge 0 \cr k=l\pmod n}} 
h_{klm}(\mu_1,\mu_2 )z^k \bar z^l \varepsilon^m,
\end{equation}
where $h_{klm}(\mu_1,\mu_2 )
=h_{lkm}^*(\mu_1,\mu_2 )$, $h_{n00}(0,0)=0$,
$ h_{220}(\mu_1,\mu_2 )\ne 0$,
\[ h_{n+1,1,0}^2(\mu_1,\mu_2 )-
4h_{220}(\mu_1,\mu_2 )h_{2n,0,0}(\mu_1,\mu_2 )\ne 0. \]

The purpose of the section is to prove that after some canonical formal transformations
Hamiltonian (\ref{Eq:hfam}) takes the form (\ref{Eq:hfinalC}). An alternative form (\ref{Eq:hfamalt})
is also derived.

As in Section~\ref{Se:hn0is0} $\delta$-degree is defined 
by (\ref{Eq:deltahn0is0}). Then
\begin{equation}\label{Eq:hfamgroop}
H(z,\bar z; \varepsilon ,\mu_1,\mu_2 )=
\sum_{p\ge 1, m\ge 0} 
h_{(p)m}(z,\bar z;\mu_1,\mu_2 ) \varepsilon^m,
\end{equation}
where $h_{(p)m}$ is homogeneous resonant polynomial
on $(z,\bar z)$ of 
$\delta$-degree $p$ with coefficients depending on
$(\mu_1,\mu_2)$. The terms of $\delta$-degree 1 are
\begin{equation}\label{Eq:h1fam}
\begin{array}{l}
h_{(1)1}(z,\bar z;\mu_1,\mu_2 )=z\bar z+h_{n01}(\mu_1,\mu_2 )z^n
+h_{n01}^*(\mu_1,\mu_2 )\bar z^n,  \\
h_{(1)m}(z,\bar z;\mu_1,\mu_2 )=h_{n0m}(\mu_1,\mu_2 )z^n
+h_{n0m}^*(\mu_1,\mu_2 )\bar z^n, \quad m\ne 1 .
\end{array}
\end{equation}

The all transformations described below 
 preserve the form of $h_{(1)m}$ 
(i.e. $h_{110}=1$ and $h_{11m}=0$ for $m\ne 1$) 
 although coefficients $h_{n0m}$
are changed.

As for individual map, normalisation of terms of 
$\delta$-degree $p=2$ is different from $p\ge 3$.
So it is convenient to consider the case of $p=2$
separately.

\begin{lemma}
Let $H(z,\bar z; \varepsilon ,\mu_1,\mu_2 )$
be a formal series (\ref{Eq:hfamgroop})
with $h_{(1)m}(z,\bar z;\mu_1,\mu_2 )$ as in 
(\ref{Eq:h1fam}) and
\[  h_{(2)m}(z,\bar z;\mu_1,\mu_2 )= 
h_{22m}(\mu_1,\mu_2 )z^2\bar z^2
+ h_{n+1,1,m}(\mu_1,\mu_2 )z^{n+1}\bar z 
 \]
\[+ h_{n+1,1,m}^*(\mu_1,\mu_2 )z\bar z^{n+1}
+ h_{2n,0,m}(\mu_1,\mu_2 )z^{2n}
+h_{2n,0,m}^*(\mu_1,\mu_2 )\bar z^{2n}. \]
If $h_{220}(\mu_1,\mu_2 )\ne 0$ and 
$h_{n+1,1,0}^2(\mu_1,\mu_2 )-
4h_{220}(\mu_1,\mu_2 )h_{2n,0,0}(\mu_1,\mu_2 )\ne 0$ then
there exists a formal canonical change of variables which
transforms the Hamiltonian $H$ 
into $\widetilde H= \sum_{p\ge 1, m \ge 0} \tilde h_{(p)m}\varepsilon^m$ with
\begin{equation}\label{Eq:nuhn0}
\tilde h_{n00}(\mu_1,\mu_2 )
=
h_{n00} (\mu_1, \mu_2)\exp \left( -\frac{i}{2}
\arg \left( h_{2n,0,0}(\mu_1, \mu_2)-\frac{h_{n+1,1,0}^2(\mu_1, \mu_2)}{4h_{220}(\mu_1, \mu_2)}   \right) \right),
\end{equation} 
 $\tilde h_{11m}(\mu_1,\mu_2 )=h_{11m}(\mu_1,\mu_2 )=0$ for $m\ne 1$,
$\tilde h_{111}(\mu_1,\mu_2 )=h_{111}(\mu_1,\mu_2 )=1$,\begin{equation}\label{Eq:h2}
\tilde h_{(2)m}(z,\bar z;\mu_1,\mu_2 )
=\tilde h_{22m}(\mu_1,\mu_2 )z^2\bar z^2 
+\tilde h_{2n,0,m}(\mu_1,\mu_2 )(z^{2n}+\bar z^{2n}),
\end{equation}
where $\tilde h_{2n,0,m}(\mu_1,\mu_2 ) \in \mathbb{R}$
for all $m \ge 0$,

$\tilde h_{2n,0,0}(\mu_1,\mu_2)=
\left| h_{2n,0,0}(\mu_1,\mu_2)-
\frac{h_{n+1,1,0}^2(\mu_1,\mu_2)}{4h_{220}(\mu_1,\mu_2)}
\right|$.
\end{lemma}

\begin{proof}
The lemma is proved by induction.
For $m=0$ we use change of variables
$(z, \bar z ) \mapsto \Phi_\chi^{1}(z, \bar z )$
with $\chi =\beta z^n +\beta^* \bar z^n$, $ \beta = -\frac{h_{n+1,1,0}}{2inh_{220}}$.
From (\ref{Eq:htransL}) we obtain 
coefficients of terms of $\delta$-degree 1:
\[\begin{array}{l}
\tilde h_{n0m}=h_{n0m} \quad {\rm for\ }m=0
{\rm \ and\ }m\ge 2 ,\\
\tilde h_{n01}=h_{n01}+in\beta 
=h_{n01}-\frac{h_{n+1,1,0}}{2h_{220}} 
\end{array}\]
and of $\delta$-degree 2, $m=0$ :
\[\begin{array}{l}
\tilde h_{220}=h_{220} \quad {\rm for\ }n\ge 4, \\
\tilde h_{220}=h_{220}-in^2\beta^*h_{300} 
\quad {\rm for\ }n=3 ,\\
\tilde h_{n+1,1,0}=h_{n+1,1,0}+2in\beta h_{220} =0,\\
\tilde h_{2n,0,0}=h_{2n,0,0}+in\beta h_{n+1,1,0}
-n^2\beta^2h_{220}=h_{2n,0,0}-\frac{h_{n+1,1,0}^2}{4h_{220}} \ne 0.
\end{array}\]
Note that $\tilde h_{220}$
is changed only for the case of resonance of the order
$n=3$ but $\tilde h_{220}(0,0)
=h_{220}(0,0)$ as $h_{300}(0,0)=0$.

After rotation 
$z \mapsto z \exp \left( -i 
\frac{\arg \tilde h_{2n,0,0}}{2n} \right) $
the coefficient $\tilde h_{2n,0,0} \mapsto |\tilde h_{2n,0,0} |$. Thus $\tilde h_{(2)0}$ has the form (\ref{Eq:h2}).

Let $h_{(2)m}$ has the form (\ref{Eq:h2})
for all $m\le M-1$. In order to normalise the term $h_{(2)M} $
we use successively two substitutions of the type
$(z,\bar z)\mapsto\Phi_\chi^{\varepsilon^M}(z,\bar z )$.
Then
\[ \widetilde H=H+\varepsilon^M L_\chi H 
+\sum_{k\ge 2} \frac{\varepsilon^{kM}}{k!}L^k_\chi H .\]
First we assume $\chi =\beta z^n +\beta^* \bar z^n$ 
and  then $\chi =\alpha z\bar z$. 
Choosing appropriate value of $\beta$ we eliminate
$h_{n+1,1,M}$ and we get $\Im h_{2n,0,M}=0$ by choosing
$\alpha $.

After the first substitution with
$ \beta = -\frac{h_{n+1,1,M}}{2inh_{220}}$
the coefficients of terms of $\delta$-degree 1 are:
\[\begin{array}{l}
\tilde h_{n0m}=h_{n0m} \quad {\rm for\ }m\ne M+1, \\
\tilde h_{n,0,M+1}=h_{n,0,M+1}+in\beta 
=h_{n,0,M+1}-\frac{h_{n+1,1,M}}{2h_{220}}
\end{array}\]
and for terms of $\delta$-degree 2 for $m \le M-1$ the coefficients are not changed and for  $m=M$:
\[\begin{array}{l}
\tilde h_{22M}=h_{22M} \quad {\rm for\ }n\ge 4 ,\\
\tilde h_{22M}=h_{22M}-in^2\beta^*h_{300} 
\quad {\rm for\ }n=3 ,\\
\tilde h_{n+1,1,M}=h_{n+1,1,M}+2in\beta h_{220}=0,\\
\tilde h_{2n,0,M}=h_{2n,0,M}+in\beta h_{n+1,1,0}=h_{2n,0,M}
-\frac{h_{n+1,1,M}h_{n+1,1,0}}{2h_{220}}.
\end{array}\]

The change $(z, \bar z ) \mapsto \Phi_\chi^{\varepsilon^M}(z, \bar z )$ with $\chi=\alpha z\bar z$ leads to 
$\tilde h_{klm} =h_{klm}$ for $m \le M-1$ and
\[ \tilde h_{klM} =h_{klM} -i\alpha (k-l)h_{kl0}  
\quad {\rm for\ }k>l.\]
So,
\[
\tilde h_{n,0,M}=h_{n,0,M}-in\alpha h_{n00} 
\]
and $\tilde h_{22M}=h_{22M}$, 
$\tilde h_{n+1,1,M}=h_{n+1,1,M}=0$,
\[
\tilde h_{2n,0,M}=h_{2n,0,M}-2in\alpha h_{2n,0,0}.
\]
Choosing $\alpha =\frac{\Im h_{2n,0,M}}{2nh_{2n,0,0}}$ 
we get $\tilde h_{2n,0,M} \in \mathbb{R}$.
\end{proof}

Now we introduce new parameters $(\nu_1 ,\nu_2 )$
instead of $(\mu_1,\mu_2 )$. Let 
\[ \tilde h_{n00}(\mu_1,\mu_2 )=\nu=\nu_1+i\nu_2,\]
where $\tilde h_{n00}(\mu_1,\mu_2 )$ is determined 
by the formula
(\ref{Eq:nuhn0}).  
Let $\mu_1$ and $\mu_2$ be expressed in terms of $(\nu_1,\nu_2 )$
as power series.
Then we get Hamiltonian $H$ in the form of the series 
in five variables $(z,\bar z ;
\varepsilon , \nu_1, \nu_2 )$.
 
Let $\Upsilon =(\varepsilon , \nu_1,\nu_2   ) $, 
$\mathbf{m}=(m_1,m_2,m_3)$, $|\mathbf{m}|=m_1+m_2+m_3$, 
$\Upsilon^{\mathbf{m}}= \varepsilon^{m_1} \nu_1^{m_2}\nu_2^{m_3}$. After collecting 
terms of the same $\delta$-degree Hamiltonian $H$ 
takes the form:
\[
H=\varepsilon z\bar z +\nu z^n +\nu^* \bar z^n
+\sum_{\substack{m_1\ge 1 \cr m_2,m_3\ge 0}}
(b_{0\mathbf{m}}z^n +b^*_{0\mathbf{m}}\bar z^n)
\Upsilon^\mathbf{m}
\]
\begin{equation}\label{Eq:hdeltanu}
+ az^2\bar z^2 +c(z^{2n }+\bar z^{2n})+
\sum_{|\mathbf{m}| \ge 1} (a_{0\mathbf{m}}z^2\bar z^2
+c_{0\mathbf{m}}(z^{2n}+\bar z^{2n}))\Upsilon^\mathbf{m}
\end{equation}
\[ +\sum_{p\ge 3, |\mathbf{m}| \ge 0} 
h_{(p)\mathbf{m}} (z,\bar z)
\Upsilon^\mathbf{m} .\]
Here the terms of $\delta$-degree 1 are in the form
(\ref{Eq:h1fam}) and the terms of $\delta$-degree 2 are already in the form
(\ref{Eq:h2}), $a=h_{220}(0,0)$, 
$c=\left| h_{2n,0,0}(0,0)-
\frac{h_{n+1,1,0}^2(0,0)}{4h_{220}(0,0)}
\right|$.

The next proposition completes transformation of Hamiltonian to the normal form. 
\begin{proposition}\label{Pro:Famhn0is0}
Let (\ref{Eq:hdeltanu})
be a formal series where $h_{(p)\mathbf{m}}$ are 
real-valued resonant polynomials on $(z, \bar z)$ 
of $\delta$-degree $p$ and $ac\ne 0$.
There exists a formal canonical change of variables which
transforms the Hamiltonian $H$
into
\begin{equation}\label{Eq:hfinalC}
\widetilde H(z,\bar z; \Upsilon )=\varepsilon z \bar z +\nu z^n+ \nu^*\bar z^n+a (z\bar{z})^{2}+c(z^{2n}+\bar z^{2n})
+(z\bar{z})^{2}\sum_{k+|\mathbf{m}|\ge 1} a_{k\mathbf{m}} {(z\bar{z})}^k \Upsilon^\mathbf{m}
\end{equation}
\begin{equation*}
+\sum_{\substack{k+m_1\ge 1 \cr m_2,m_3\ge 0}} {(z\bar{z})}^{2k} (b_{k\mathbf{m}}z^{n}+b_{k\mathbf{m}}^*\bar z^n) \Upsilon^\mathbf{m}
+ (z^{2n}+\bar z^{2n})  \sum_{k+|\mathbf{m}|\ge 1} c_{k\mathbf{m}} {(z\bar{z})}^{2k} \Upsilon^\mathbf{m} ,
\end{equation*}
where $a_{k\mathbf{m}}, c_{k\mathbf{m}} \in \mathbb{R}$, 
$b_{k\mathbf{m}} \in \mathbb{C}$.
\end{proposition}

\begin{proof}
Let
$\delta$-degree of monomial $z^k\bar z^l$
is defined by (\ref{Eq:deltahn0is0}) as before.
Now we introduce a new $\delta_{\Upsilon}$-degree for monomial on five variables:
\begin{equation}\label{Eq:deltaepsilonorder}
\delta_{\Upsilon} (z^k \bar z^l \Upsilon^\mathbf{m})=\delta(z^k \bar z^l)+2|\mathbf{m}| .
\end{equation}
Then
\[ H(z, \bar z; \Upsilon )= \sum_{s \ge 2} h_s(z, \bar z; \Upsilon ), \]
where $h_s$ is a homogeneous polynomial of 
$\delta_{\Upsilon}$-degree $s$. The main term of 
$\delta_{\Upsilon}$-degree $s=2$
is already in the normal form:
\[ h_2(z, \bar z; \Upsilon ) = a (z\bar{z})^{2}+c(z^{2n}+\bar z^{2n}) .\]

We proceed by induction. Let $S \ge 3$ 
and $h_s(z, \bar z; \Upsilon )$ has the declared form
for $s \le S-1$.

The term of $\delta_{\Upsilon}$-degree $S$ has the form:
\begin{equation}
\label{Eq:hS}
 h_S (z, \bar z; \Upsilon )=
\sum_{0 \le |\mathbf{M}|   \le \lfloor \frac{S-1}{2}\rfloor} h_{(S-2|\mathbf{M}|)\mathbf{M}} (z,\bar z)
\Upsilon^\mathbf{M},
\end{equation}
where $h_{(S-2|\mathbf{M}|)\mathbf{M}} (z,\bar z)$
is a polynomial of $\delta$-degree $S-2|\mathbf{M}|$.
The terms of $\delta$-degree 1 and 2 are already
in the normal form. So we consider $S-2|\mathbf{M}|\ge 3$,
i.e. $0\le |\mathbf{M}|\le \lfloor \frac{S-3}{2}\rfloor$.

Let $\chi$  be a homogeneous polynomial on $(z,\bar z)$ 
 of $\delta$-degree $P-1=S-2|\mathbf{M}|-1$. 
After the substitution $(z, \bar z ) \mapsto \Phi_\chi^{\Upsilon^{\mathbf{M}}}(z, \bar z )$
the Hamiltonian takes the form
\[
\widetilde H = H+\Upsilon^{\mathbf{M}} L_\chi H + \sum_{j \ge 2} \frac{\Upsilon^{\mathbf{M}j}}{j!} L_\chi^j H \, .
\]
Then
\[ \tilde h_s (z, \bar z; \Upsilon ) = h_s (z, \bar z; \Upsilon ) \quad \mbox{for \ }s \le S-1 \]
and
\[
\tilde h_S (z, \bar z; \Upsilon ) = h_S (z, \bar z; \Upsilon ) + \Upsilon^{\mathbf{M}}L(\chi) ,
\]
where $L$ is the homological operator (\ref{Eq:homop}) from Section \ref{Se:hn0is0}.
So
\[\tilde h_{(P)\mathbf{m}}=h_{(P)\mathbf{m}} \quad {\rm for\ }
\mathbf{m}\ne \mathbf{M}\]
and
\[\tilde h_{(P)\mathbf{M}}=h_{(P)\mathbf{M}}+L(\chi ).\]

It was been shown in Section \ref{Se:hn0is0} 
that for odd $P$ there exists such $\chi$ that
\[
\tilde h_{(P)\mathbf{M}} =a_{P\mathbf{M}}z^P\bar z^P+
{(z\bar z )}^{P-1}(
b_{P\mathbf{M}} z^{n}+b_{P\mathbf{M}}^* \bar z^{n})
\]
and for even $P$
\[\tilde h_{(P)\mathbf{M}}=a_{P\mathbf{M}}z^P\bar z^P+
c_{P\mathbf{M}} {(z\bar z )}^{P-2}
(z^{2n}+\bar z^{2n}) ,\]
where $a_{P\mathbf{M}},c_{P\mathbf{M}} \in \mathbb{R}$, $b_{P\mathbf{M}} \in \mathbb{C}$.
So, taking sequentially all sets of $(m_1,m_2,m_3)$
such that $0\le m_1+m_2+m_3\le \lfloor \frac{S-3}{2}\rfloor$ we transform all terms of 
$\delta_{\Upsilon}$-degree $S$
to the declared form.
\end{proof}


The following proposition establishes the normal form of Hamiltonian useful for study of bifurcations.
\begin{proposition}\label{Pro:Famhnois0}
Let (\ref{Eq:hdeltanu})
be a formal series where $a\ne 0$ and
$h_{(p)\mathbf{m}}$ are 
real-valued resonant polynomials of $\delta$-degree $p$.
There exists a formal canonical change of variables which
transforms the Hamiltonian $H$
into
\begin{equation}\label{Eq:hfamalt}
\begin{array}{r}
\tilde H(z,\bar z; \Upsilon )=\varepsilon z \bar z +a (z\bar{z})^{2}+(z\bar{z})^{2}\sum_{k+|\mathbf{m}|\ge 1} a_{k\mathbf{m}} {(z\bar{z})}^k \Upsilon^\mathbf{m}   \\
+\nu z^n+ \nu^*\bar z^n + \sum_{k\ge 1,k+m_1\ge 2,
|\mathbf{m}|\ge 0}  (c_{k\mathbf{m}}z^{kn}+c_{k\mathbf{m}}^*\bar z^{kn}) \Upsilon^\mathbf{m}
\end{array}
\end{equation}
where $a_{k\mathbf{m}}\in \mathbb{R}$, $c_{k\mathbf{m}}\in \mathbb{C} $,  $c_{2\mathbf{0}}=c$.
\end{proposition}

\begin{proof}
We use $\delta_{\Upsilon}$-degree introduced by
\[ \delta_{\Upsilon}(z^k\bar z^l\Upsilon^{\mathbf{m}})=k+l+2|{\mathbf{m}}|\]
 and normalized order by order as in Proposition~\ref{Pro:Famhn0is0} but
now we choose the compliment to the image of the homological operator as in the proof of Proposition~\ref{Pro:hn0is0case1}.
\end{proof}


\section{Bifurcations for the case of degeneracy in the leading resonant term}\label{Se:biflead}

We discuss typical level sets of Hamiltonian (\ref{Eq:hfamalt})
for $0<|\varepsilon |+|\nu |\le \epsilon_0$ with sufficiently small $\epsilon_0$. 
In order to investigate biffurcations we keep lower order terms 
and skip those, which do not change the picture qualitatively:
\[ H(z, \bar z ; \varepsilon , \tilde \nu )=\varepsilon z\bar z+a (z\bar{z})^{2}
+\tilde \nu z^n+\tilde \nu^* \bar z^n
+c(z^{2n}+\bar z^{2n}),\]
where $\tilde \nu =\nu+\varepsilon c_{1100}+\dots +\varepsilon^kc_{1k00}$, 
$k=\left\lfloor \frac{n}{2}\right\rfloor$.

Applying the symplectic polar coordinates 
(\ref{Eq:zpolar}) and assuming $\tilde \nu =\frac{1}{2}
\gamma e^{i\beta}$ we get
the model Hamiltonian in the form
\begin{equation}\label{Eq:modelp_B}
H(I, \varphi; \varepsilon, \gamma , \beta )
=\varepsilon I+ I^2+\gamma I^{n/2}\cos (n\varphi+\beta)+I^n\cos 2n\varphi\,.
\end{equation}
Here the coefficients $a$ and $c$ are normalized to unity 
with the help of a scaling applied to the variable $I$,
the parameters $\varepsilon$ and $\gamma$, and 
the Hamiltonian function $H$.

The case of $n=3$ is different
from the case of $n\ge 4$.

\subsection{Typical level sets of the model
Hamiltonian for $n\ge 4$}

Since 
\begin{equation}\label{Eq:prIH}
\partial_I H = \varepsilon +2I+O(\gamma I^{n/2-1})
+O(I^{n-1}),
\end{equation}
when $\varepsilon >0$  for $n \ge 4$
Hamiltonian $H$ 
has not any critical points near the origin and
 level sets are closed curves.

Typical level sets for $\varepsilon <0$
are depending on value of parameters $\varepsilon $ and $\gamma$. We consider separatly the case of 
$\gamma \gg |\varepsilon |^{n/2}$
and $\gamma = O(|\varepsilon |^{n/2})$.

If ${|\varepsilon |}^{n/2}=o(\gamma )$
then the last term in (\ref{Eq:modelp_B})
can be omitted. 
Then critical points of Hamiltonian 
are located near $I=-\frac{\varepsilon}{2}$,
$\cos (n\varphi+\beta)=\pm 1$. 
 There is a chain of $n$ islands
on the distance of order $|\varepsilon |^{1/2}$ from the origin.
The size of islands is of order 
$\gamma^{1/2} |\varepsilon |^{n/4}$.

 For
 $\gamma = O\left({|\varepsilon | }^{n/2}\right)$ the last term in (\ref{Eq:modelp_B})
 is also essential.
In order to study bifurcations we introduce the
scaling:
\begin{equation}\label{Eq:scaling4p_B}
\gamma=4{\left( -\frac{\varepsilon}{2}\right) }^{n/2} b,\quad
I=-\frac{\varepsilon}{2}+{\left( -\frac{\varepsilon}{2}\right) }^{n/2}J,\quad \psi =n\varphi ,\quad 
H={\left( -\frac{\varepsilon}{2}\right) }^n\bar H\,.
\end{equation}


After this scaling the Hamiltonian takes the form (scipping the constant term)
$\bar H(J,\psi)=\bar H_0(J,\psi)+
O\left( {\left( -\frac{\varepsilon}{2}\right) }^{n/2-1}
\right) $, where
\begin{equation}\label{Eq:barh0_4p_B}
\bar H_0(J,\psi)=J^2+4b\cos (\psi+\beta)+\cos 2\psi\,.
\end{equation}
Equilibrium points of $\bar H_0(J,\psi)$ locate at the points for which $J=0$ and
\begin{equation}\label{Eq:hyperb}
b\sin (\psi+\beta)+\sin \psi \cos \psi =0.
\end{equation}
To figure out how many solutions has this equation we denote $\sin \psi =x$ and $\cos \psi =y$. Then equation (\ref{Eq:hyperb})
is equvalent to the system
\[
\left\lbrace {
\begin{array}{c} xb\cos \beta  +yb \sin \beta  +xy =0 \\ x^2 +y^2=1
\end{array}
} \right. 
\] 
First equation corresponds to hyperbole (or 2 straight lines if $\cos \beta=0$, or $\sin \beta =0$, or $b=0$). One branch of the hyperbole (or one of stright lines) passes through the origin.
So at least two solutions exist for arbitrary 
$b$ and $\beta$.

\begin{figure}[t]
\begin{center}
\includegraphics[width=5cm]{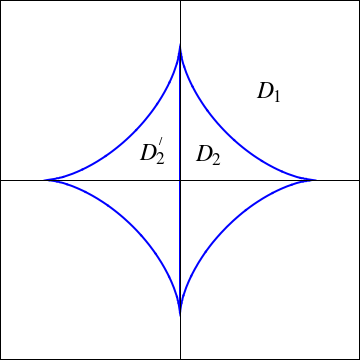}
\end{center}
\caption{
Bifurcation diagram on the complex plane of 
$\nu=\gamma e^{i\beta }$ for 
$\varepsilon <0$. The domains $D_2$ and $D_2'$ are 
separated by the vertical line segment.
\label{Fig:nbbifdiag}}
\end{figure}

There are two more solutions if second branch of hyperbole (or second straight line) crosses 
the unit circle. Tangency condition for the unit circle and the second brunch of 
 the hyperbole  is
\[ {(b\cos \beta)}^{2/3}+{(b \sin \beta)}^{2/3}=1. \]
On the complex plane of $\gamma e^{i\beta}
=\nu =\nu_1+i\nu_2$ corresponding line is given by
astroid:
\begin{equation}\label{Eq:astr}
{\nu_1}^{2/3}+{\nu_2}^{2/3}=2^{4/3}{\left( -\frac{\varepsilon}{2}\right) }^{n/3} \, 
\end{equation}
and is presented on Fig.~\ref{Fig:nbbifdiag}. 
Typical critical level sets of Hamiltonian 
corresponding different values of parameter $\nu$ 
(provided $\varepsilon <0$) are shown on Fig.~\ref{Fig:h0ll}.
Fragments on the figure are repeated $n$ times.
On the boundary between $D_2$ and $D_2'$ 
(vertical line segment)
there is a chain of $2n$ islands (Fig.~\ref{Fig:h0ll}(c)).
Outside a domain bounded by the astroid there is a chain of $n$ islands.

\begin{figure}[t]
\begin{center}
\includegraphics[width=2.9cm]{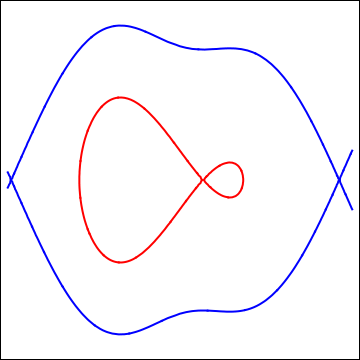}\quad
\includegraphics[width=2.9cm]{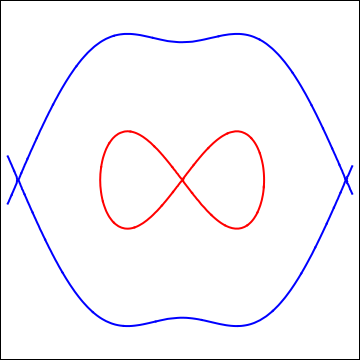}\quad
\includegraphics[width=2.9cm]{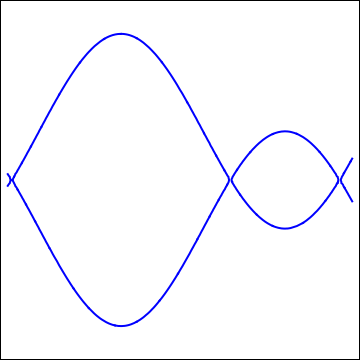}\quad
\includegraphics[width=2.9cm]{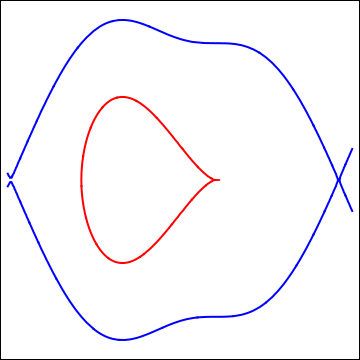}\\
(a)\kern2.9cm(b)\kern2.9cm(c)\kern2.9cm(d)
\end{center}
\caption{Critical level sets of the 
model Hamiltonian for $\varepsilon <0$ depending on value of $\nu$ indicated on bifurcation diagram on Fig.~\ref{Fig:nbbifdiag}: (a) $\nu \in D_2$, (b) $\nu \in D_2$, $ \beta =0$, (c) $\nu$ is in the boundary between $D_2$ and $D_2'$ ($\beta =\pm \pi /2$), (d) $\nu$ is in the boundary between $D_2$ and $D_1$.
\label{Fig:h0ll}}
\end{figure}

\subsection{Typical level sets of model Hamiltonian
for $n=3$}

For $\varepsilon <0$ the case of $n=3$ essentially differs
from the case of $n\ge 4$ as the third term in (\ref{Eq:prIH})
cannot be omitted.

The first three terms in (\ref{Eq:prIH})
are of the same order if $I$ and $\varepsilon $
are of the order of $\gamma^2$. So we 
use the following scaling in (\ref{Eq:modelp_B})
\begin{equation}\label{Eq:scaling3b1} 
\varepsilon =a \gamma^2, \quad I=\gamma^2 J, \quad 
H=\gamma^4 \bar H \,  
\end{equation}
and get $\bar H = \bar H_0 +O(\gamma^2)$, where
\begin{equation}\label{Eq:n3sc1} \bar H_0 =a J+J^{3/2}\cos (3\varphi +\beta ) +J^2. 
\end{equation}

\begin{figure}[t]
\begin{center}
\includegraphics[width=5cm]{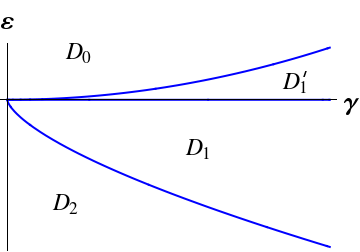}
\end{center}
\caption{
Bifurcation diagram on the 
plane $(\gamma ,\varepsilon )$ for 
$n=3$.
\label{Fig:n3bifdiag}}
\end{figure}

\begin{figure}[t]
\begin{center}
\includegraphics[width=3cm]{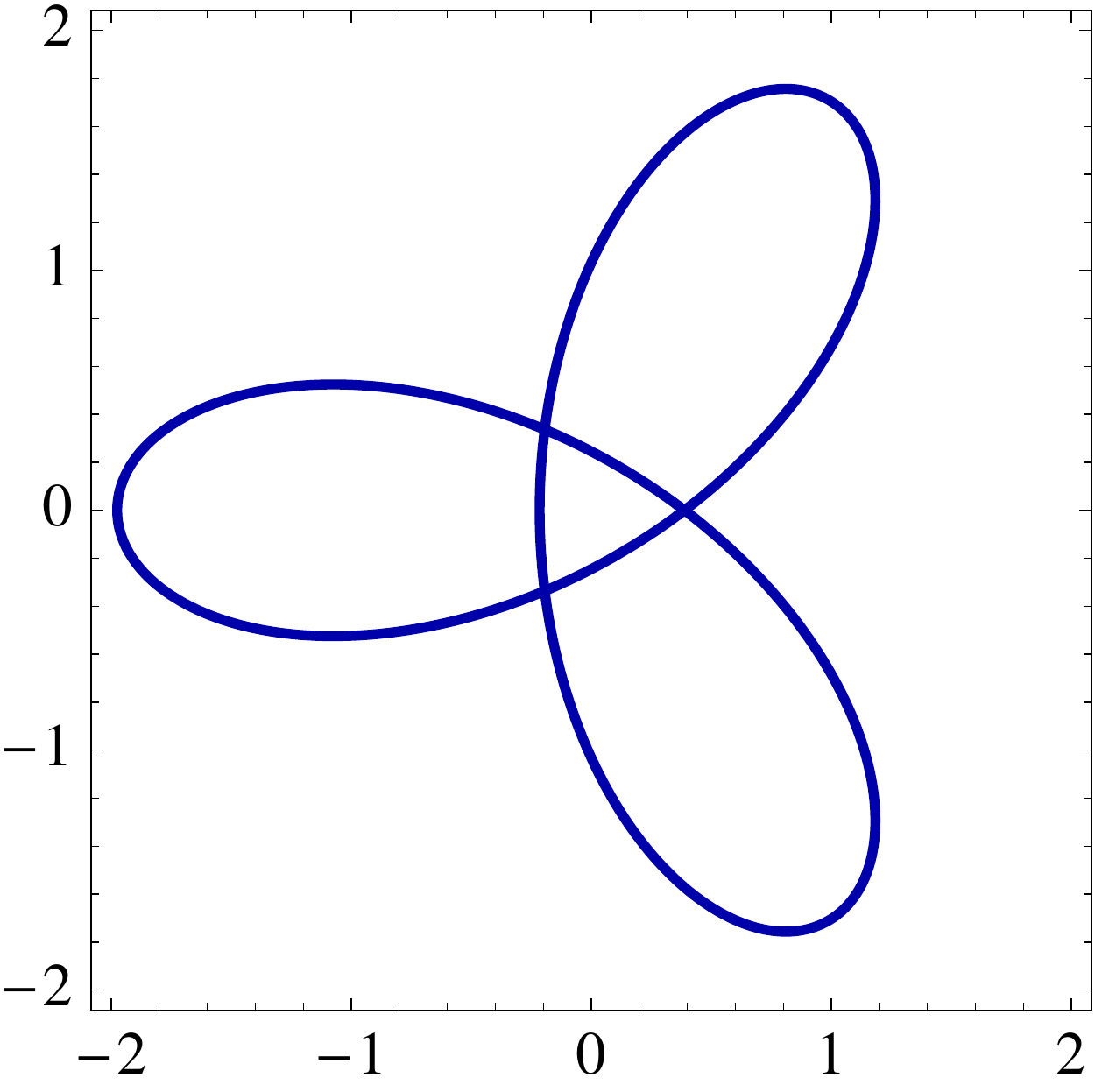}
\kern0.2cm
\includegraphics[width=3cm]{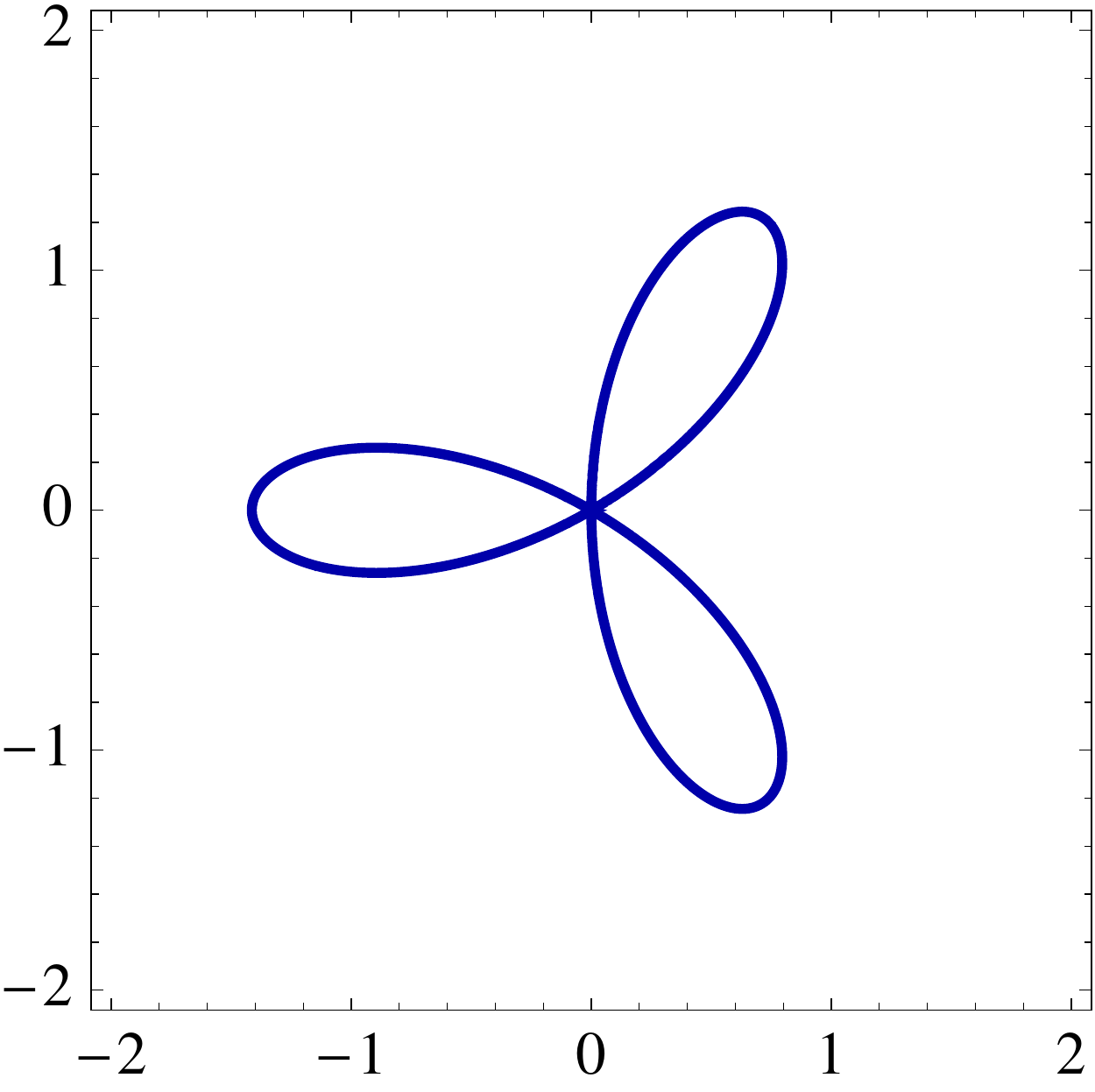}
\kern0.2cm
\includegraphics[width=3cm]{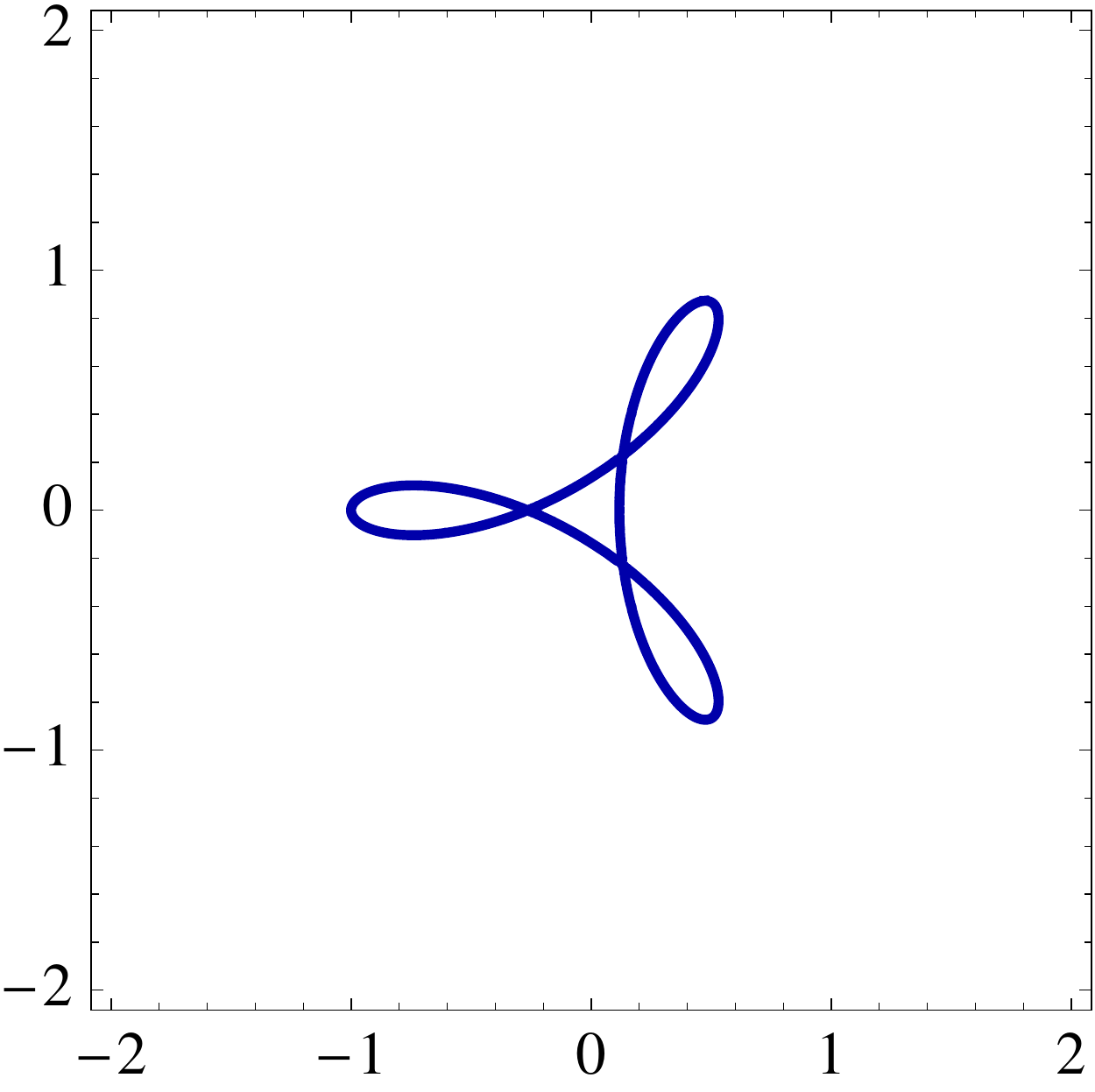}
\kern0.2cm
\includegraphics[width=3cm]{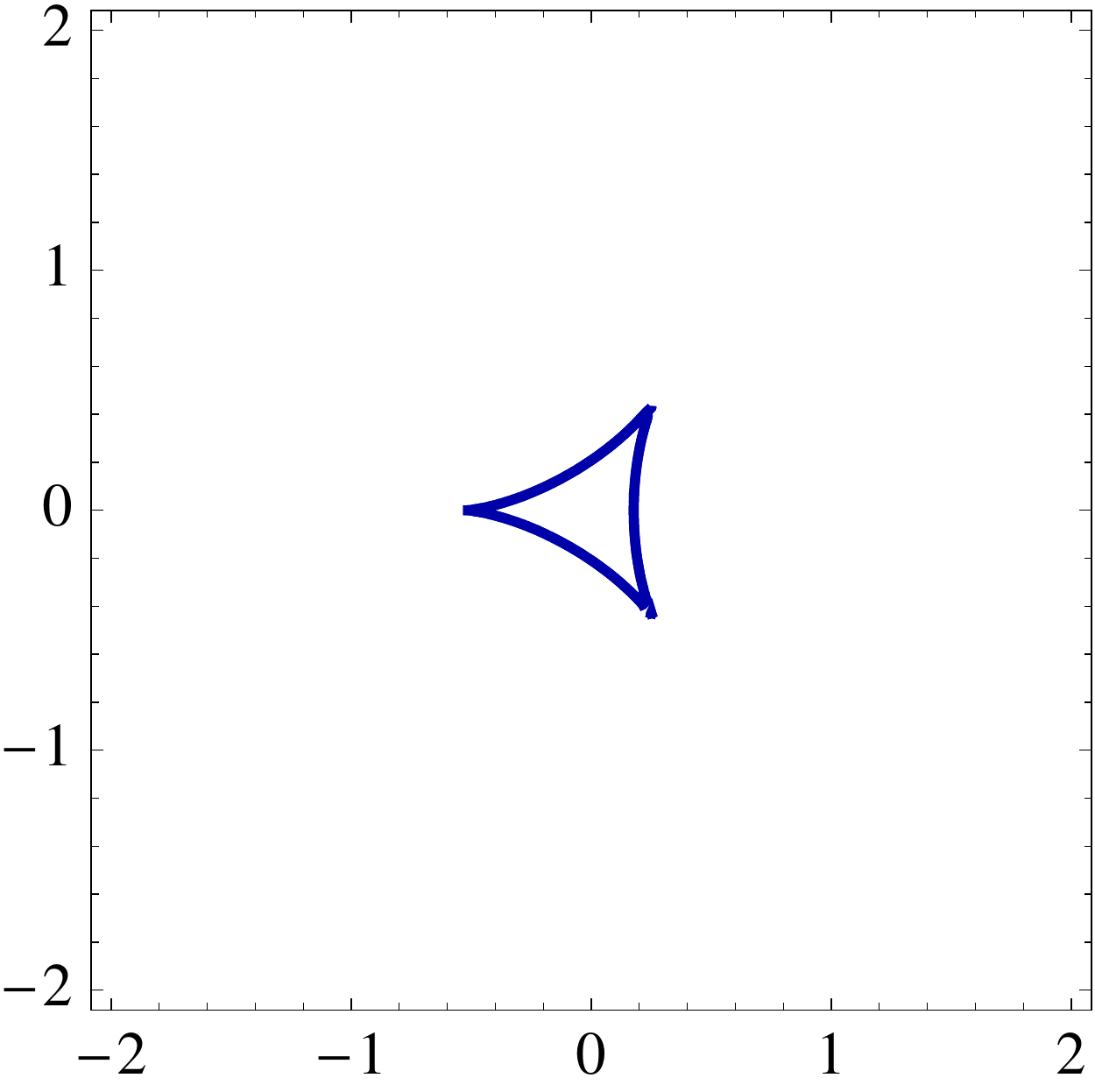}
\\
(a)\kern2.8cm%
(b)\kern2.8cm%
(c)\kern2.8cm%
(d)%
\end{center}
\caption{Critical level sets of the model Hamiltonian depending on value of $(\gamma , \varepsilon )$:
(a)~$(\gamma , \varepsilon )\in D_1$, 
(b)~the boudary between $D_1$ and $D_1'$ ($\varepsilon =0$), 
(c)~$(\gamma , \varepsilon )\in D_1'$, 
(d)~the boudary between $D_1'$ and $D_0$
($\varepsilon =\frac{9}{32} \gamma^2$).
\label{Fig:n3b2}
}
\end{figure}

The critical points of $\bar H_0$ are located
at the points for which $\cos(3\varphi +\beta)=\sigma_\varphi=\pm 1$,
and
$$
a +\frac{3\sigma_\varphi}2J^{1/2}+2J=0\,.
$$
The last equation can be solved explicitly.  
There are no real solutions for $a>a_0$, where $a_0=\frac{9}{32}$. If $a\in (0;a_0)$ there are two
solutions for $\sigma_\varphi =-1$ and no one for
$\sigma_\varphi =1$. And if $a<0$ there is one solution 
for $\sigma_\varphi =-1$ and one for
$\sigma_\varphi =1$.

The line
$$
\varepsilon =\frac{9}{32}\gamma^2\,.
$$
on the plane $(\gamma, \varepsilon )$
bounded the domain in which level sets of
the Hamiltonian (\ref{Eq:n3sc1}) are closed curves
(the domain $D_0$ on Fig.~\ref{Fig:n3bifdiag}).
 
The critical level sets of the Hamiltonian $\bar H_0$ are illustrated on Figure~\ref{Fig:n3b2}. 

For $\varepsilon <0$ the case of $n=3$ is not different from $n\ge 4$ (see the analysis in the previous subsection). In particular, the boundary between $D_1$ and $D_2$ is given by (\ref{Eq:astr}). 
On the Fig.~\ref{Fig:n3bifdiag}  we assume 
$\nu_1=\nu =\gamma $, $\nu_2=0$ which correspond to
$\beta =0$ ($\nu =\gamma e^{i\beta }$). Then the boundary between $D_1$ and $D_2$ is given by 
\[ \varepsilon = -2^{-1/3} \gamma^{2/3}. \]

For $(\gamma , \varepsilon ) \in D_2$ level sets of the model Hamiltonian are presented on Fig.~\ref{Fig:h0ll}.


\section{Families with twist degeneracy\label{Se:Familiesh33}}

The theorem about normal forms for families 
with twist degeneracy can be found in \cite{GG2014}. For three-parametric families it takes the following form.

\begin{proposition} \label{Pro:h33issmall}
Let 
$\Upsilon =(\varepsilon_1,\varepsilon_2,\varepsilon_3 )$,
$\mathbf{m} = (m_1,m_2,m_3 )$,
$\Upsilon^\mathbf{m} = \varepsilon_1^{m_1}
\varepsilon_2^{m_2}\varepsilon_3^{m_3}$.
If  
\begin{equation}\label{Eq:h_originalh33small}
H(z,\bar z; \Upsilon)=\varepsilon_1 z\bar z+\varepsilon_2 {(z\bar z)}^2+\varepsilon_3 {(z\bar z)}^3+\sum_{\substack {k+l > 2 ,\ k,l,|\mathbf{m}|\ge 0, \\k=l \pmod n}}  h_{kl\mathbf{m}}z^k \bar z^l\Upsilon^\mathbf{m} 
\end{equation}
is a formal series such that $h_{kl\mathbf{m}}=h_{lk\mathbf{m}}^*$, $h_{22\mathbf{m}}=h_{33\mathbf{m}}=0$  and $h_{44\mathbf{0}}h_{n0\mathbf{0}}\ne 0$,
then there exists a formal tangent-to-identity canonical change of variables which
transforms the Hamiltonian $H$
into
\begin{eqnarray*}\label{Eq:hn0is0}
\widetilde H(z,\bar z;\Upsilon)=\varepsilon_1 z\bar z+\varepsilon_2 {(z\bar z)}^2+\varepsilon_3 {(z\bar z)}^3+
(z\bar{z})^{4}\sum_{k, |\mathbf{m}|\ge 0} a_{k\mathbf{m}} {(z\bar{z})}^k \Upsilon^\mathbf{m} \\
+(z^n+\bar z^n)\sum_{k, |\mathbf{m}|\ge 0} b_{k\mathbf{m}} {(z\bar{z})}^k \Upsilon^\mathbf{m},
\end{eqnarray*}%
where $a_{k\mathbf{m}},b_{k\mathbf{m}} \in \mathbb{R}$,  $b_{0\mathbf{0}}=|h_{n0\mathbf{0}}|$. Moreother
\begin{itemize}
\item
if $3 \le n \le 8$

$a_{k\mathbf{m}}=0$ for $k=n-5 \pmod n$ and $b_{k\mathbf{m}}=0$ for $k=n-1 \pmod n$ 

\item
if $n\ge 8$

$b_{k\mathbf{m}}=0$ for $k=3 \pmod 4$

\end{itemize}
The coefficients $a_{k\mathbf{m}}$ and $b_{k\mathbf{m}}$ are defined uniquely.
\end{proposition}

In the symplectic polar coordinates
(\ref{Eq:zpolar}) the model Hamiltonian is
\begin{equation} \label{Eq:modelHtwist} H(I,\varphi ;\Upsilon)=\varepsilon_1I +\varepsilon_2 I^2+\varepsilon_3I^3+I^4+b_0I^{n/2} \cos n\varphi \, .
\end{equation}
Its typical level sets depend on which of two last terms
is the main. If the term $I^4$ is smaller then the last 
term ($n \le 7$ or $n=8$ and $|b_0|>1$) then for
$\Upsilon =0$ the origin is not stable. 
If $n \ge 9$ or $n=8$ and 
$|b_0|<1$ then
the origin is stable for $\Upsilon =0$. 
Bifurcations in stable and unstable cases are briefly
considered below.

\subsection{Stable case }

Let $n\ge 9$. For $0\le I\ll 1$, 
the Hamiltonian (\ref{Eq:modelHtwist}) can be considered as a small perturbation
of $H_0=\varepsilon_1I +\varepsilon_2 I^2+\varepsilon_3I^3+I^4$. The level sets of $H_0$ are circles for all values of the
parameters. The Hamiltonian (\ref{Eq:modelHtwist}) does not posses  this symmetry. 
The implicit function theorem implies that
the last term can affect on the shape of level sets
of $H$
only near zeroes of $\partial_I 
H_0$. The equation 
\begin{equation}\label{Eq:bdroots}
\partial_IH_0=\varepsilon_1 +2\varepsilon_2 I+3\varepsilon_3I^2+4I^3 =0
\end{equation} 
has from $0$ to $3$ solutions depending on $(\varepsilon_1,\varepsilon_2,\varepsilon_3 )$.

\begin{figure}[t]
\begin{center}
\includegraphics[width=3.8cm]{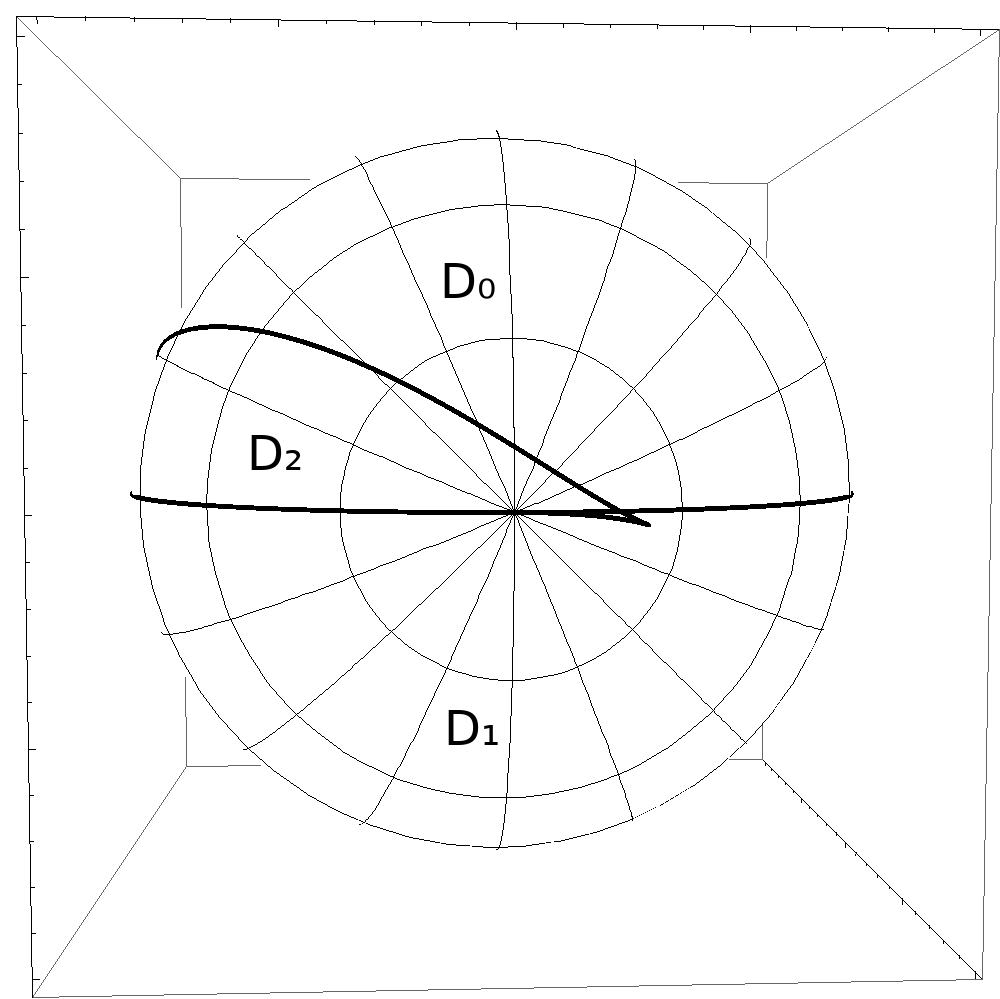}\qquad
\includegraphics[width=3.8cm]{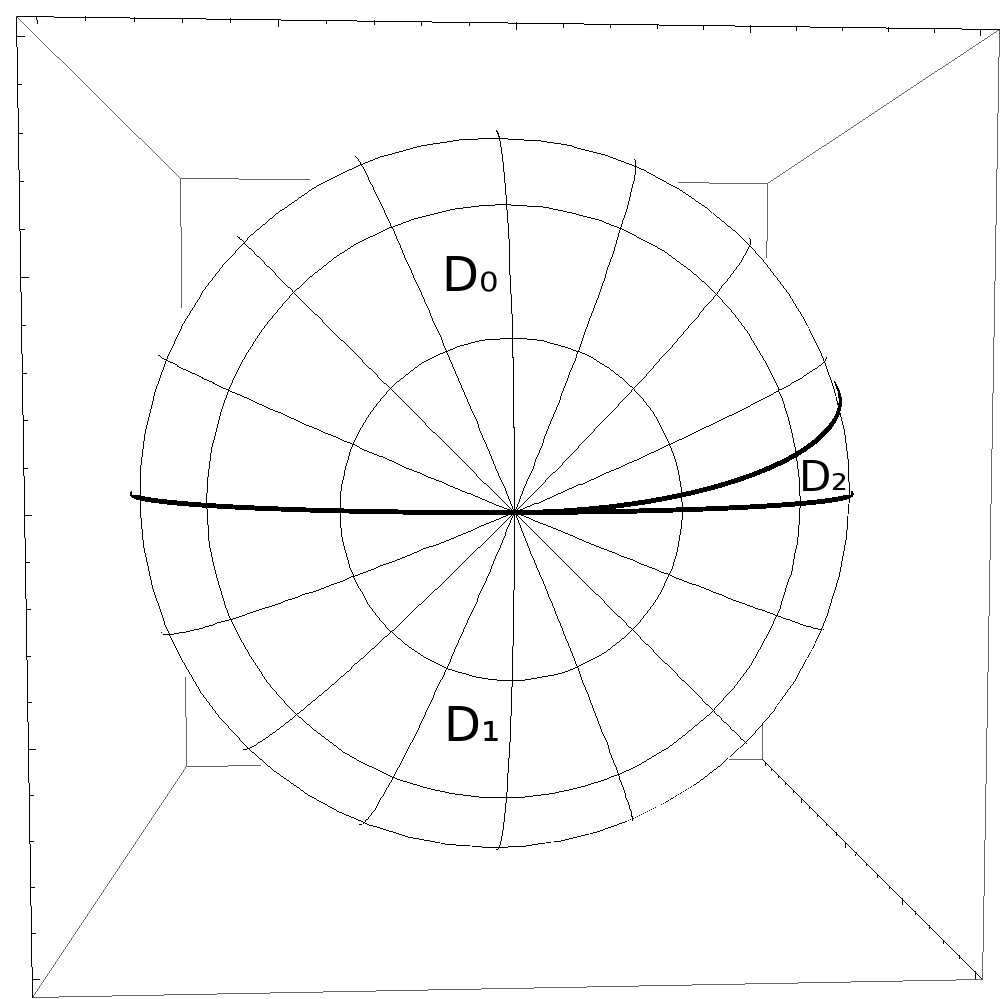}\\
(a)\kern3.8cm(b)
\end{center}
\caption{Bifurcation diagram for stable case on the
unit sphere in the space of parameters $(\varepsilon_1,
\varepsilon_2,\varepsilon_3)$. 
\label{Fig:bdtwist}}
\end{figure}

Let $D_k$ be a domain in the parameter space such that
 equation (\ref{Eq:bdroots}) has $k$ roots, {\it i.e.}
 $H_0$ has $2kn$ stationary points.
These domains on 
the unit sphere in the space 
$(\varepsilon_1, \varepsilon_2,\varepsilon_3)$ are shown
on the Figure~\ref{Fig:bdtwist}.

In $D_0$ all level sets of $H$ 
are closed curves around the origin. In $D_1$ there is 
one chain of islands. In $D_2$ typical level sets are 
similar to ones in two-parametric families \cite{GG2014}.

\begin{figure}[t]
\begin{center}
\includegraphics[width=2.9cm]{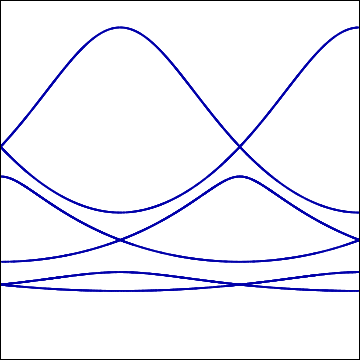}\quad
\includegraphics[width=2.9cm]{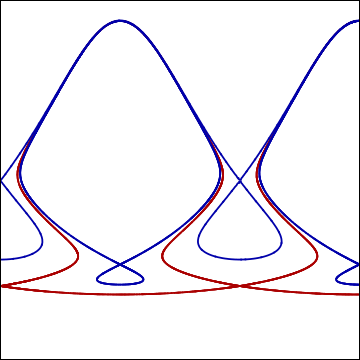}\quad
\includegraphics[width=2.9cm]{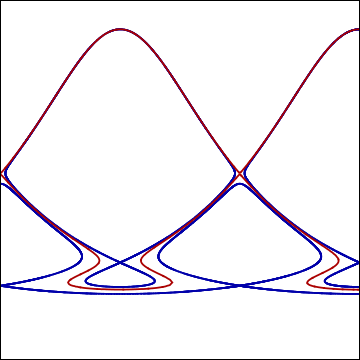}\quad
\includegraphics[width=2.9cm]{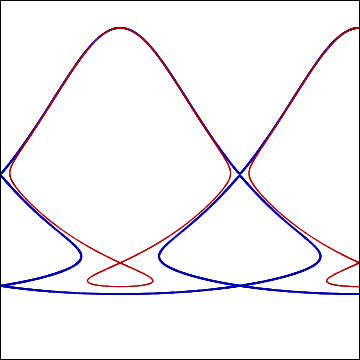}\\
(a)\kern2.9cm(b)\kern2.9cm(c)\kern2.9cm(d)
\end{center}
\caption{The typical critical level sets for the Hamiltonian (\ref{Eq:modelHtwist}) in the stable case if parameters $(\varepsilon_1,\varepsilon_2,\varepsilon_3) \in D_3$. 
\label{Fig:Stable}}
\end{figure}

In neighbourhood of the point $\varepsilon_1=0$, 
$\varepsilon_2=0$, $\varepsilon_3=-1$ there is a tiny domain $D_3$. When parameters are in $D_3$
the equation (\ref{Eq:bdroots}) have 3 roots. 
They correspond to three sets of hyperbolic 
and three sets of elliptic stationary points
of the model Hamiltonian (\ref{Eq:modelHtwist}). Some of the possible corresponding critical level sets are shown on Figure~\ref{Fig:Stable}. The fragment shown on the figure is repeated $n$ times around the origin.

\begin{figure}[t]
\begin{center}
\includegraphics[width=3.8cm]{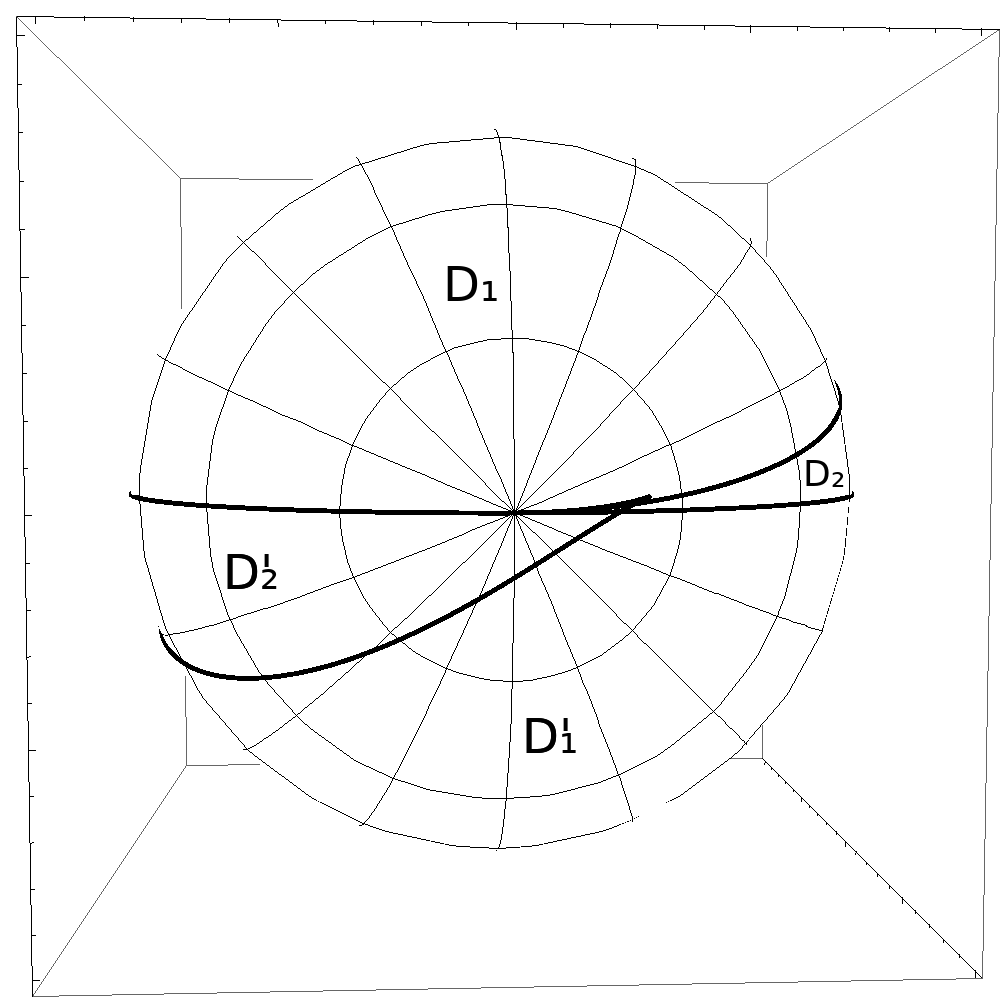}\quad
\includegraphics[width=3.8cm]{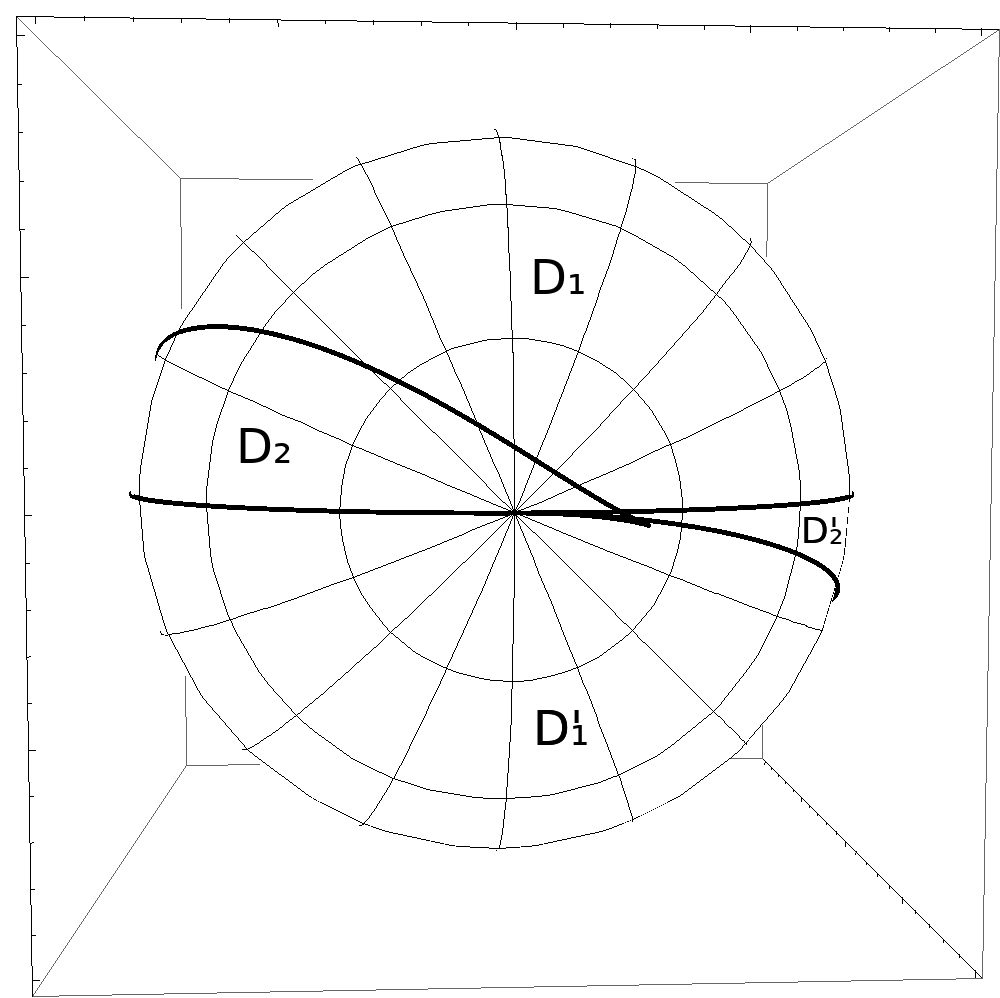}\\
(a)\kern3.8cm(b)
\end{center}
\caption{Bifurcation diagram for unstable case on the
unit sphere in the space of parameters $(\varepsilon_1,
\varepsilon_2,\varepsilon_3)$. 
\label{Fig:bdtwistunst}}
\end{figure}

\subsection{Unstable case}

\begin{figure}[t]
\begin{center}
\includegraphics[width=3.8cm]{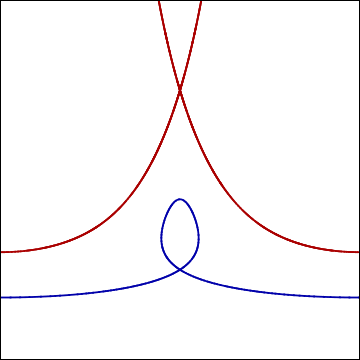}\quad
\includegraphics[width=3.8cm]{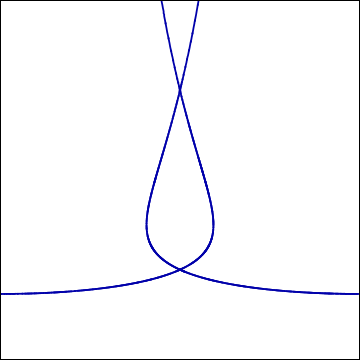}
\quad
\includegraphics[width=3.8cm]{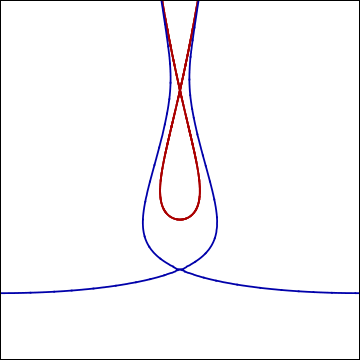}\\
(a)\kern3.8cm(b)\kern3.8cm(c)
\end{center}
\caption{Some critical level sets for $n=8$, $b_0=2$. 
\label{Fig:bdunstln}}
\end{figure}

For $n \le 7$ or $n=8$ and $|b_0|>1$ the bifurcation diagram on the unit sphere in the space of parameters 
$(\varepsilon_1,\varepsilon_2,\varepsilon_3)$ is shown in Figure~\ref{Fig:bdtwistunst}.
In $D_1$ and $D_1'$ typical level sets are the same as in one-parametric families.
In $D_2$ and $D_2'$ typical level sets and their bifurcations  are essentially the same as in 
two-parametric families. In a neighbourhood
of the points $\varepsilon_1=\varepsilon_2=0$,
$\varepsilon_3 = \pm 1$ there are tiny domains
with additional sets of stationary points.
Some possible critical
level sets for this case are shown 
on Figure~\ref{Fig:bdunstln}.



\end{document}